\documentclass[12pt]{amsart}
\usepackage{amsmath}
\usepackage{amsfonts, amssymb, euscript, mathrsfs}
\usepackage{amsthm, upref}
\usepackage{graphicx}
\usepackage[usenames, dvipsnames]{color}

\usepackage{tikz}
\usepackage{enumerate}
\usepackage[tmargin=1in,bmargin=1in,lmargin=0.8in,rmargin=0.8in]{geometry}
\usepackage{aliascnt}
\usepackage{hyperref}
\usepackage{verbatim}

\newcommand{\MZ}{\mathbb{Z}}

\newcommand{\BR}{\mathbb{R}}

\newcommand{\SL}{\sum\limits}

\newcommand{\al}{\alpha}
\newcommand{\be}{\beta}
\newcommand{\ga}{\gamma}
\newcommand{\de}{\delta}

\newcommand{\La}{\Lambda}
\newcommand{\ME}{\mathbf E}

\newcommand{\oa}{\omega}
\newcommand{\CF}{\mathcal F}

\newcommand{\CU}{\mathcal U}
\newcommand{\MP}{\mathbf P}

\newcommand{\CN}{\mathcal N}

\newcommand{\CX}{\mathcal X}

\newcommand{\CW}{\mathcal W}

\newcommand{\Oa}{\Omega}
\newcommand{\si}{\sigma}

\newcommand{\pa}{\partial}
\renewcommand{\phi}{\varphi}

\newcommand{\eps}{\varepsilon}
\newcommand{\la}{\lambda}
\newcommand{\Ra}{\Rightarrow}

\newcommand{\ol}{\overline}

\newcommand{\norm}[1]{\lVert#1\rVert}
\renewcommand{\comment}[1]{}

\newcommand{\mP}{\mathbf{p}}
\newcommand{\md}{\mathrm{d}}

\DeclareMathOperator{\Var}{Var}

\DeclareMathOperator{\diag}{diag}
\DeclareMathOperator{\const}{const}

\DeclareMathOperator{\mes}{mes}

\DeclareMathOperator{\Exp}{Exp}
\DeclareMathOperator{\CE}{Exp}

\begin{document}

\theoremstyle{plain}
\newtheorem{thm}{Theorem}[section]
\newtheorem*{thmnonumber}{Theorem}
\newtheorem{lemma}[thm]{Lemma}
\newtheorem{prop}[thm]{Proposition}
\newtheorem{cor}[thm]{Corollary}
\newtheorem{open}[thm]{Open Problem}

\theoremstyle{definition}
\newtheorem{defn}{Definition}
\newtheorem{asmp}{Assumption}
\newtheorem{notn}{Notation}
\newtheorem{prb}{Problem}

\theoremstyle{remark}
\newtheorem{rmk}{Remark}
\newtheorem{exm}{Example}
\newtheorem{clm}{Claim}

\author{Andrey Sarantsev}

\title{Infinite Systems of Competing Brownian Particles}

\address{Department of Statistics and Applied Probability, University of California, Santa Barbara}

\email{sarantsev@pstat.ucsb.edu}



\date{September 3, 2016. Version 47}

\begin{abstract}
Consider a system of infinitely many Brownian particles on the real line. At any moment, these particles can be ranked from the bottom upward. Each particle moves as a Brownian motion with drift and diffusion coefficients depending on its current rank. The gaps between consecutive particles form the (infinite-dimensional) gap process. We find a stationary distribution for the gap process. We also show that if the initial value of the gap process is stochastically larger than this stationary distribution, this process converges back to this distribution as time goes to infinity.  This continues the work by Pal and  Pitman (2008). Also, this includes infinite systems with asymmetric collisions, similar to the finite ones from Karatzas, Pal and Shkolnikov (2016). 
\end{abstract}

\maketitle

\section{Introduction}

Consider the standard setting: a filtered probability space $(\Oa, \CF, (\CF_t)_{t \ge 0}, \MP)$, with the filtration satisfying the usual conditions. Take i.i.d. $(\CF_t)_{t \ge 0}$-Brownian motions $W_i = (W_i(t), t \ge 0)$, $i = 1, 2, \ldots$ Consider an infinite system $X = (X_i)_{i \ge 1}$ of real-valued adapted processes $X_i = (X_i(t), t \ge 0),\, i = 1, 2, \ldots$, with $\MP$-a.s. continuous trajectories. Suppose we can rank them in the increasing order at every time $t \ge 0$: 
$$
X_{(1)}(t) \le X_{(2)}(t) \le \ldots 
$$
If there is a {\it tie}: $X_i(t) = X_j(t)$ for some $i < j$ and $t \ge 0$, we assign a lower rank to $X_i$ and higher rank to $X_j$. Now, fix coefficients $g_1, g_2, \ldots \in \BR$ and $\si_1, \si_2, \ldots > 0$. Assume each process $X_i$ (we call it a {\it particle}) moves according to the following rule: if at time $t$ $X_i$ has rank $k$, then it evolves as a one-dimensional Brownian motion with drift coefficient $g_k$ and diffusion coefficient $\si_k^2$. Letting $1(A)$ be the indicator function of an event $A$, we can write this as the following system of SDEs:
\begin{equation}
\label{eq:main-SDE-intro}
\md X_i(t) = \SL_{k=1}^{\infty}1\left(X_i\ \mbox{has rank}\ k\ \mbox{at time}\ t\right)\left(g_k\md t + \si\md W_i(t)\right),\ i = 1, 2, \ldots
\end{equation}
The {\it gaps} $Z_k(t) = X_{(k+1)}(t) - X_{(k)}(t)$ for $k = 1, 2, \ldots$ form the {\it gap process} $Z = (Z(t), t \ge 0)$, $Z(t) = (Z_k(t))_{k \ge 1}$. Then $X$ is called an {\it infinite system of competing Brownian particles}. A more precise definition is given in Definitions~\ref{main} and ~\ref{defapprox} later in this article. 

This system was studied in \cite{Shkolnikov2011, IKS2013}. For $g_1 = 1, g_2 = g_3 = \ldots = 0$ and $\si_1 = \si_2 = \ldots = 1$, this is called the {\it infinite Atlas model}, which was studied in \cite{PP2008, DemboTsai}.  The term {\it Atlas} stands for the bottom particle, which moves as a Brownian motion with drift $1$ (as long as it does not collide with other particles) and ``supports other particles on its shoulders''. This system is, in fact, a generalization of a similar finite system $X = (X_1, \ldots, X_N)'$, which is defined analogously to the equation~\eqref{eq:main-SDE-intro}. Finite systems of competing Brownian particles were originally introduced in \cite{BFK} as a model in Stochastic Portfolio Theory, see \cite{FernholzBook, FernholzKaratzasSurvey}. They also serve as scaling limits for exclusion processes on $\MZ$, see \cite[Section 3]{KPS2012}, and as a discrete analogue of McKean-Vlasov equation, which governs a nonlinear diffusion process, \cite{S2012, 4people, JR2013a}. Finite systems were thoroughly studied recently. We can ask the following questions about them: 

\medskip

(a) Does this system exist in the weak or strong sense? Is it unique in law or pathwise?

(b) Do we have triple collisions between particles, when three or more particles occupy the same position at the same time?

(c) Does the gap process have a stationary distribution? Is it unique?

(d) What is the exact form of this stationary distribution? 

(e) Does $Z(t)$ converge weakly to this stationary distribution as $t \to \infty$?

\medskip

For finite systems, these questions have been to a large extent answered. 

\medskip

(a) The system exists in the weak sense and is unique in law, \cite{BassPardoux}. Until the first moment of a triple collision, it exists in the strong sense and is pathwise unique, \cite{IKS2013}. It is not known whether it exists in the strong sense after this first triple collision. 

\medskip

(b) It was shown in \cite{IK2010, IKS2013, MyOwn3} that there are a.s. no triple collisions if and only if the sequence $(\si_1^2, \ldots, \si_N^2)$ is {\it concave}: 
\begin{equation}
\label{eq:concave}
\si_k^2 \ge \frac12\left(\si_{k-1}^2 + \si_{k+1}^2\right),\ \ k = 2, \ldots, N-1.
\end{equation}

\medskip

(c) The gap process has a stationary distribution if and only if 
\begin{equation}
\label{eq:stability}
\ol{g}_k > \ol{g}_N,\ k = 1, \ldots, N-1,\ \ \mbox{where}\ \ \ol{g}_k := \frac1k\left(g_1 + \ldots + g_k\right)\ \ \mbox{for}\ \ k = 1, \ldots, N.
\end{equation}
In this case, this stationary distribution is unique, see \cite{BFK, 5people}.

\medskip

(d) If, in addition to~\eqref{eq:stability}, the sequence $(\si_1^2, \ldots, \si_N^2)$ is {\it linear}: 
\begin{equation}
\label{eq:linearity}
\si_{k+1}^2 - \si_k^2 = \si_k^2 - \si_{k-1}^2\ \ \mbox{for}\ \ k = 2, \ldots, N-1,
\end{equation}
then this stationary distribution has a product-of-exponentials form, see \cite{BFK, 5people}.

\medskip

(e) The answer is affirmative, under the condition~\eqref{eq:stability}, see \cite{5people, WilliamsSurvey, Chen1996}. 

\medskip

Before surveying the answers for infinite systems, let us define some notation. Let $N \in \{1, 2, \ldots\}\cup\{\infty\}$. Introduce a componentwise (partial) order on $\BR^{N}$. Namely, take $x = (x_i)$ and $y = (y_i)$ from $\BR^{N}$. For $M \le N$, we let $[x]_M := (x_i)_{i \le M}$. For a distribution $\pi$ on $\BR^N$, we let $[\pi]_M$ be the marginal distribution on $\BR^M$, corresponding to the first $M$ components. For a matrix $C = (c_{ij})_{i, j \le N}$, we let $[C]_M = (c_{ij})_{i, j \le M}$. We say that $x \le y$ if $x_i \le y_i$ for all $i \ge 1$. For $x \in \BR^{N}$, we let $[x, \infty) := \{y \in \BR^{N}\mid y \ge x\}$. We say that two probability measures $\nu_1$ and $\nu_2$ on $\BR^{N}$ satisfy $\nu_1 \preceq \nu_2$, or, equivalently, $\nu_2 \succeq \nu_1$, if for every $y \in \BR^{\infty}$ we have: $\nu_1[y, \infty) \le \nu_2[y, \infty)$. In this case, we say that $\nu_1$ {\it is stochastically dominated by} $\nu_2$, and $\nu_2$ {\it stochastically dominates} $\nu_1$, or $\nu_1$ is {\it stochastically smaller} than $\nu_2$, or $\nu_2$ is {\it stochastically larger} than $\nu_1$. We denote weak convergence of probability measures by $\nu_n \Ra \nu$. We denote by $I_k$ the $k\times k$-identity matrix. For a vector $x = (x_1, \ldots, x_d)' \in \BR^d$, let $\norm{x} := \left(x_1^2 + \ldots + x_d^2\right)^{1/2}$ be its Euclidean norm. For any two vectors $x, y \in \BR^d$, their dot product is denoted by $x\cdot y = x_1y_1 + \ldots + x_dy_d$. The Lebesgue measure is denoted by $\mes$. A one-dimensional Brownian motion with zero drift and unit diffusion, starting from $0$, is called a {\it standard Brownian motion}. Let 
$$
\Psi(u) := \frac1{\sqrt{2\pi}}\int_u^{\infty}e^{-v^2/2}\md v,\ \ u \in \BR,
$$
be the tail of the standard normal distribution. 

\medskip

For infinite systems, the answers to questions (a) - (e) are quite different. 

\medskip

(a) For infinite systems, it seems that a necessary condition for weak existence is that initial positions $X_i(0) = x_i$, $i = 1, 2, \ldots$ of the particles should be ``far apart''. Indeed, it is an easy exercise to show that a system of i.i.d standard Brownian motions starting from the same point is not rankable from bottom to top at any fixed time $t > 0$. Some sufficient conditions for weak existence and uniqueness in law are found in \cite{Shkolnikov2011, IKS2013}. We restate them in Theorem~\ref{existenceinfclassical} in a slightly different form:
\begin{equation}
\label{seriesalpha}
\lim\limits_{i \to \infty}x_i = \infty\ \ \mbox{and}\ \ \SL_{i=1}^{\infty}e^{-\al x_i^2} < \infty,\ \ \al > 0.
\end{equation}
We also prove a few other similar results: Theorem~\ref{thm:Girsanov-existence} and Theorem~\ref{convnames}, under slightly different conditions. Strong existence and pathwise uniqueness for finite systems are known from \cite{IKS2013} to hold until the first {\it triple collision}, when three or more particles simltaneously occupy the same position. It is not known whether these hold after this first triple collision.

\medskip

(b) In this paper, we continue on the research in \cite{IKS2013} and prove essentially the same result as for finite systems. There are a.s. no triple collisions if and only if the sequence $(\si_k^2)_{k \ge 1}$ is concave: see Theorem~\ref{thmasymminf} and Remark~\ref{rmk:symm-3ple}. 

\medskip

(c) In this paper, see Theorem~\ref{thm3}, we prove that there exists a certain stationary distribution $\pi$ under the condition which is very similar to~\eqref{eq:stability}:
\begin{equation}
\label{eq:stability-inf}
\ol{g}_k > \ol{g}_l,\ \ 1 \le k < l.
\end{equation}
Actually, we can even relax this condition~\eqref{eq:stability-inf} a bit, see~\eqref{eq:comp-of-g-bar}. The question whether it is unique or not is still open.

\medskip

(d) The exact form of this distribution $\pi$ is found in~\eqref{eq:pi-SS-symm} for a special case~\eqref{eq:linearity}; it is also a product of exponentials, as in the finite case.

\medskip

(e) We prove a partial convergence result in Theorem~\ref{limitpointthm} and Theorem~\ref{thm:main-conv}: if $Z(0)$ stochastically dominates this stationary distribution $\pi$: $Z(0) \succeq \pi$, then $Z(t) \Ra \pi$ as $t \to \infty$. However, we do not know whether $Z(t)$ weakly converges as $t \to \infty$ for other initial distributions. Since we do not know whether a stationary distribution is unique, this means that we do not know what are the ``domains of attraction''.

\medskip

Let us give a preview of results for a special case: 
\begin{equation}
\label{eq:M-Atlas}
g_1 = g_2 = \ldots = g_M = 1,\ g_{M+1} = g_{M+2} = \ldots = 0,\ \si_1 = \si_2 = \ldots = 1.
\end{equation}

The following theorem is a corollary of more general results (which are enumerated above) from this paper; see Example 4.2 below. 

\begin{thm}
\label{thm1}
Under conditions~\eqref{eq:M-Atlas}, the system~\eqref{eq:main-SDE-intro} exists in the strong sense, is pathwise unique, there are a.s. no triple and simultaneous collisions, and the stationary distribution $\pi$ for the gap process is given by
\begin{equation}
\label{eq:pi-M}
\pi_M := \Exp(2)\otimes\Exp(4)\otimes\ldots\otimes\Exp(2M)\otimes\Exp(2M)\otimes\ldots
\end{equation}
\end{thm}

For $M = 1$, this is the infinite Atlas model, and the stationary distribution $\pi_M = \pi_1 = \otimes_{k=1}^{\infty}\Exp(2)$ is already known from \cite[Theorem 14]{PP2008}. It is worth noting that the \textsc{Harris} system of Brownian particles (independent Brownian motions $B_n, n \in \MZ$, starting from $B_n(0) = x_n$), in fact, has infinitely many stationary distributions for its gap process, \cite{Harris}. Indeed, a Poisson point process with constant intensity $\la$ is invariant with respect to this system for any $\la > 0$. Therefore, the product $\otimes_{n \in \MZ}\Exp(\la)$ is a stationary distribution for this system, for all $\la > 0$. 

We also direct the reader to our paper \cite{MyOwn13}, which is complementary to the current paper. In \cite{MyOwn13}, we find other stationary distributions for the gap process. Instead of stating the main result, we consider the particular case of the system~\eqref{eq:main-SDE-intro}. There, for every $a > 0$, the following is a stationary distribution for the gap process:
\begin{equation}
\label{eq:weird}
\pi_M(a) := \bigotimes\limits_{k=1}^{\infty}\Exp\left(2(k\wedge M) + ka\right).
\end{equation}
In particular, for the infinite Atlas model we have:
$$
\pi_1(a) := \bigotimes\limits_{k=1}^{\infty}\Exp\left(2 + ka\right).
$$
Note that the distribution~\eqref{eq:pi-M} can also be included in the family~\eqref{eq:weird}, if we let $a = 0$.

Other ordered particle systems derived from independent driftless Brownian motions were studied by \textsc{Arratia} in \cite{Arratia}, and by \textsc{Sznitman} in \cite{Sznitman1986}. Several other papers study connections between systems of queues and one-dimensional interacting particle systems: \cite{Kipnis, Ferrari1, Ferrari2, Timo}. Links to the directed percolation and the directed polymer models, as well as the GUE random matrix ensemble, can be found in \cite{OConnellYor}.

An important generalization of a finite system of competing Brownian particles is a {\it system with asymmetric collisions}, when, roughly speaking, ranked particles $Y_k$, have ``different mass'', and when they collide, they ``fly apart'' with ``different speed''. This generalization was introduced in \cite{KPS2012} for finite systems. We carry out this generalization for infinite systems, and prove weak existence (but not uniqueness) in Section 3. All results answering the questions (a) - (e) above are stated also for this general case of asymmetric collisions. 

There are other generalizations of competing Brownian particles: competing L\'evy particles, \cite{Shkolnikov2011}; {\it a second-order stock market model}, when the drift and diffusion coefficients depend on the name as well as the rank of the particle, \cite{2order, 5people}; competing Brownian particles with values in the positive orthant $\BR^N_+$, see \cite{FIKP2013}. Two-sided infinite systems $(X_i)_{i \in \MZ}$ of competing Brownian particles are studied in \cite{MyOwn11}. 

The proofs in this article rely heavily on {\it comparison techniques} for systems of competing Brownian particles, developed in \cite{MyOwn2}. 
%

\subsection{Organization of the paper} Section 2 is devoted to the necessary background: finite systems of competing Brownian particles. It does not contain any new results, just an outline of already known results. Section 3 introduces infinite systems of competing Brownian particles and states existence and uniqueness results (including Theorem~\ref{thm2}). In this section, we also generalize these comparison techniques for infinite systems. Section 4 deals with the gap process: stationary distributions and the questions of weak convergence as $t \to \infty$. In particular, we state Theorems~\ref{thm3} and~\ref{limitpointthm} and in this section. Section 5 contains results about triple collisions. Section 6 is devoted to proofs for most of the results. The Appendix contains some technical lemmas.

\section{Background: finite systems of competing Brownian particles}

In this section, we recall definitions and results which are already known.  First, as in \cite{BFK, MyOwn3}, we rigorously define finite systems of competing Brownian particles for the case of symmetric collisions, when the $k$th ranked particles moves as a Brownian motion with drift coefficient $g_k$ and diffusion coefficient $\si_k^2$. This gives us a system of {\it named} particles; we shall call them {\it classical systems of competing Brownian particles}. Then we find an equation for corresponding {\it ranked} particles, following \cite{BFK, 5people}. This gives us a motivation to introduce systems of ranked competing Brownian particles with asymmetric collisions, as in \cite{KPS2012}. Finally, we state results about the gap process: stationary distribution and convergence.

\subsection{Classical systems of competing Brownian particles} 

In this subsection, we use definitions from \cite{BFK}. Assume the usual setting: a filtered probability space $(\Oa, \CF, (\CF_t)_{t \ge 0}, \MP)$ with the filtration satisfying the usual conditions.  Let $N \ge 2$ (the number of particles). Fix parameters 
$$
g_1, \ldots, g_N \in \BR;\ \ \si_1, \ldots, \si_N > 0.
$$
We wish to define a system of $N$ Brownian particles in which the $k$th smallest particle moves according to a Brownian motion with drift $g_k$ and diffusion $\si_k^2$. We resolve ties in the lexicographic order, as described in the Introduction. 

\begin{defn} Take i.i.d. standard $(\CF_t)_{t \ge 0}$-Brownian motions $W_1, \ldots, W_N$. For a continuous $\BR^N$-valued process $X = (X(t),\ t \ge 0)$, $X(t) = (X_1(t), \ldots, X_N(t))'$, let us define $\mP_t,\ t \ge 0$, the {\it ranking permutation} for the vector $X(t)$: this is the permutation on $\{1, \ldots, N\}$ such that:

\medskip

(i) $X_{\mP_t(i)}(t) \le X_{\mP_t(j)}(t)$ for $1 \le i < j \le N$; 

\medskip

(ii) if $1 \le i < j \le N$ and $X_{\mP_t(i)}(t) = X_{\mP_t(j)}(t)$, then $\mP_t(i) < \mP_t(j)$. 

\medskip

Suppose the process $X$ satisfies the following SDE:
\begin{equation}
\label{mainSDE}
dX_i(t) = \SL_{k=1}^N1(\mP_t(k) = i)\left[g_k\, \md t + \si_k\, \md W_i(t)\right],\ \ i = 1, \ldots, N.
\end{equation}
Then this process $X$ is called a {\it classical system of $N$ competing Brownian particles} with {\it drift coefficients} $g_1, \ldots, g_N$ and {\it diffusion coefficients} $\si_1^2, \ldots, \si_N^2$. For $i = 1, \ldots, N$, the component $X_i = (X_i(t), t \ge 0)$ is called the {\it $i$th named particle}. For $k = 1, \ldots, N$, the process
$$
Y_k = (Y_k(t),\ t \ge 0),\ \ Y_k(t) := X_{\mP_t(k)}(t) \equiv X_{(k)}(t),
$$
is called the {\it $k$th ranked particle}. They satisfy
$Y_1(t) \le Y_2(t) \le \ldots \le Y_N(t)$, $t \ge 0$. If $\mP_t(k) = i$, then we say that the particle $X_i(t) = Y_k(t)$ at time $t$ has {\it name} $i$ and {\it rank} $k$. 
\label{classical}
\end{defn}

The coefficients of the SDE~\eqref{mainSDE} are piecewise constant functions of $X_1(t), \ldots, X_N(t)$; therefore, weak existence and uniqueness in law for such systems follow from \cite{BassPardoux}. 

\subsection{Asymmetric collisions}

In this subsection, we consider the model defined in \cite{KPS2012}: {\it finite systems of competing Brownian particles with asymmetric collisions}. For $k = 2, \ldots, N$, let the process $L_{(k-1, k)} = (L_{(k-1, k)}(t),\ t \ge 0)$ be the semimartingale local time at zero of the nonnegative semimartingale $Y_k - Y_{k-1}$. For notational convenience, we let $L_{(0, 1)}(t) \equiv 0$ and $L_{(N, N+1)}(t) \equiv 0$. Then for some i.i.d. standard Brownian motions $B_1, \ldots, B_N$, the ranked particles $Y_1, \ldots, Y_N$ satisfy the following equation:
\begin{equation}
\label{symmSDE}
Y_k(t) = Y_k(0) + g_kt + \si_kB_k(t) + \frac12L_{(k-1, k)}(t) - \frac12L_{(k, k+1)}(t),\ \ k = 1, \ldots, N.
\end{equation}
This was proved in \cite[Lemma 1]{5people}; see also \cite[Section 3]{BFK}. The process $L_{(k-1, k)}$ is called the {\it local time of collision between the particles $Y_{k-1}$ and $Y_k$}. The local time process $L_{(k-1, k)}$ has the following properties: $L_{(k-1, k)}(0) = 0$, $L_{(k-1, k)}$ is nondecreasing, and 
\begin{equation}
\label{localtime}
\int_0^{\infty}1(Y_{k}(t) \ne Y_{k-1}(t))\md L_{(k-1, k)}(t) = 0.
\end{equation}
If we change coefficients $1/2$ in~\eqref{symmSDE} to some other real numbers, we get the model with {\it asymmetric collisions} from the paper \cite{KPS2012}. The local times in this new model are split unevenly between the two colliding particles, as if these particles have different mass. 

Let us now formally define this model with asymmetric collisions. Let $N \ge 2$ be the quantity of particles. Fix real numbers $g_1, \ldots, g_N$ and positive real numbers $\si_1, \ldots, \si_N$, as before. In addition, fix real numbers $q^+_1$, $q^-_1, \ldots, q^+_N$, $q^-_N$, which satisfy the following conditions:
$$
q^+_{k+1} + q^-_k = 1,\ \ k = 1, \ldots, N-1;\ \ 0 < q^{\pm}_k < 1,\ \ k = 1, \ldots, N.
$$

\begin{defn} Take i.i.d. standard $(\CF_t)_{t \ge 0}$-Brownian motions $B_1, \ldots, B_N$. Consider a continuous adapted $\BR^N$-valued process 
$$
Y = (Y(t),\ t \ge 0),\ \ \ Y(t) = (Y_1(t), \ldots, Y_N(t))',
$$
and $N-1$ continuous adapted real-valued processes
$$
L_{(k-1, k)} = (L_{(k-1, k)}(t),\ t \ge 0),\ \ k = 2, \ldots, N,
$$
with the following properties:

\medskip

(i) $Y_1(t) \le \ldots \le Y_N(t),\ \ t \ge 0$;

\medskip

(ii) the process $Y$ satisfies the following system of equations:
\begin{equation}
\label{SDEasymm}
Y_k(t) = Y_k(0) + g_kt + \si_kB_k(t) + q^+_kL_{(k-1, k)}(t) - q^-_kL_{(k, k+1)}(t),\ \ \ k = 1, \ldots, N
\end{equation}
(we let $L_{(0, 1)}(t) \equiv 0$ and $L_{(N, N+1)}(t) \equiv 0$ for notational convenience);

\medskip

(iii) for each $k = 2, \ldots, N$, the process $L_{(k-1, k)} = (L_{(k-1, k)}(t),\ t \ge 0)$ has the properties mentioned above: $L_{(k-1, k)}(0) = 0$, $L_{(k-1, k)}$ is nondecreasing and satisfies~\eqref{localtime}.  

\medskip

Then the process $Y$ is called {\it a system of $N$ competing Brownian particles with asymmetric collisions}, with {\it drift coefficients} $g_1, \ldots, g_N$, {\it diffusion coefficients} $\si_1^2, \ldots, \si_N^2$, and {\it parameters of collision} $q^{\pm}_1,\ldots, q^{\pm}_N$. For each $k = 1, \ldots, N$, the process $Y_k = (Y_k(t), t \ge 0)$ is called the {\it $k$th ranked particle}. For $k = 2, \ldots, N$, the process $L_{(k-1, k)}$ is called the {\it local time of collision between the particles $Y_{k-1}$ and $Y_k$}. 
\label{asymmdefn} The Brownian motions $B_1, \ldots, B_N$ are called {\it driving Brownian motions} for this system $Y$. The process 
$L = \left(L_{(1, 2)}, \ldots, L_{(N-1, N)}\right)'$ is called the {\it vector of local times}. 
\end{defn}

The state space of the process $Y$ is $\mathcal W^N := \{y = (y_1, \ldots, y_N)' \in \BR^N\mid y_1 \le y_2 \le \ldots \le y_N\}$. Strong existence and pathwise uniqueness for $Y$ and $L$ are proved in \cite[Section 2.1]{KPS2012}.  

\subsection{The gap process for finite systems} The results of this subsection are taken from \cite{BFK, 5people, KPS2012, WilliamsSurvey}. However, we present an outline of proofs in Section 6 for completeness. 

\begin{defn} Consider a finite system (classical or ranked) of $N$ competing Brownian particles. Let 
$$
Z_k(t) = Y_{k+1}(t) - Y_k(t),\ \ k = 1, \ldots, N-1,\ \ t \ge 0.
$$
Then the process $Z = (Z(t), t \ge 0)$, $Z(t) = (Z_1(t), \ldots, Z_{N-1}(t))'$ is called the {\it gap process}. The component $Z_k = (Z_k(t), t \ge 0)$ is called the {\it gap between the $k$th and $k+1$st ranked particles}. 
\end{defn}

The following propositions about the gap process are already known. We present them in a slightly different form than that from the sources cited above; for the sake of completeness, we present short outlines of their proofs in Section 6. Let  
\begin{equation}
\label{R}
R = 
\begin{bmatrix}
1 & -q^-_2 & 0 & 0 & \ldots & 0 & 0\\
-q^+_2 & 1 & -q^-_3 & 0 & \ldots & 0 & 0\\
0 & -q^+_3 & 1 & -q^-_4 & \ldots & 0 & 0\\
\vdots & \vdots & \vdots & \vdots & \ddots & \vdots & \vdots\\
0 & 0 & 0 & 0 & \ldots & 1 & -q^-_{N-1}\\
0 & 0 & 0 & 0 & \ldots & -q^+_{N-1} & 1
\end{bmatrix},
\end{equation}
\begin{equation}
\label{mu}
\mu = \left(g_2 - g_1, g_3 - g_2, \ldots, g_N - g_{N-1}\right)'.
\end{equation}

\begin{prop} (i) The matrix $R$ is invertible, and $R^{-1} \ge 0$, with strictly positive diagonal elements $\left(R^{-1}\right)_{kk},\ k = 1, \ldots, N-1$. 

\medskip

(ii) The family of random variables $Z(t), t \ge 0$, is tight in $\BR^{N-1}_+$, if and only if $R^{-1}\mu < 0$. In this case, for every initial distribution of $Y(0)$ we have: 
$Z(t) \Ra \pi$ as $t \to \infty$, where $\pi$ is the unique stationary distribution of $Z$. 

\medskip

(iii) If, in addition, the {\upshape skew-symmetry condition} holds: 
\begin{equation}
\label{SS}
(q^-_{k-1} + q^+_{k+1})\si_k^2 = q^-_k\si^2_{k+1}  + q^+_k\si_{k-1}^2,\ \ k = 2,\ldots, N-1,
\end{equation}
then 
$$
\pi = \bigotimes\limits_{k=1}^{N-1}\CE(\la_k),\ \ \la_k = \frac{2}{\si_k^2 + \si_{k+1}^2}\left(-R^{-1}\mu\right)_k,\ \ k = 1, \ldots, N-1.
$$\label{basic}
\end{prop}

For symmetric collisions, we can refine Proposition~\ref{basic}. Recall the notation from~\eqref{eq:stability}:
$$
\ol{g}_k := \frac{g_1 + \ldots + g_k}k,\ \ k = 1, \ldots, N.
$$

\begin{prop} \label{basicsymm} For the case of symmetric collisions $q^{\pm}_k = 1/2,\ k = 1, \ldots, N$, we have:

\medskip

(i) $-R^{-1}\mu = 2\left(g_1 - \ol{g}_N, g_1 + g_2 - 2\ol{g}_N, \ldots, g_1 + g_2 + \ldots + g_{N-1} - (N-1)\ol{g}_N\right)'$;

\medskip

(ii) the tightness condition from Proposition~\ref{basic} can be written as
$$
\ol{g}_k > \ol{g}_N,\ \ k = 1, \ldots, N-1;
$$

\medskip

(iii) the skew-symmety condition can be equivalently written as 
$$
\si_{k+1}^2 - \si_k^2 = \si_k^2 - \si_{k-1}^2,\ \ k = 2, \ldots, N-1;
$$
in other words, $\si_k^2$ must linearly depend on $k$;

\medskip

(iv) if both the tightness condition and the skew-symmetry condition are true, then 
$$
\pi = \bigotimes\limits_{k=1}^{N-1}\CE(\la_k),\ \ \la_k := \frac{4k}{\si_k^2 + \si_{k+1}^2}\left(\ol{g}_k - \ol{g}_N\right).
$$
\end{prop}

\begin{exm} If $g_1 = 1, \ g_2 = g_3 = \ldots = g_N = 0$, and $\si_1 = \si_2 = \ldots = \si_N = 1$ (the {\it finite Atlas model} with $N$ particles), then \label{exmatlas}
$$
\pi = \bigotimes\limits_{k=1}^{N-1}\CE\left(2\cdot\frac{N-k}N\right).
$$
\end{exm}

The following is a technical lemma, with a (very short) proof in Section 6.

\begin{lemma} Take a finite system of competing Brownian particles (either classical or ranked). For every $t > 0$, the probability that there is a tie at time $t$ is zero.
\label{notie}
\end{lemma}

\section{Existence and Uniqueness Results for Infinite Systems} 

In this section, we first state existence results for classical infinite systems of competing Brownian particles (recall that {\it classical} means particles with individual names rather than ranks): Theorem~\ref{existenceinfclassical}, Theorem~\ref{thm:Girsanov-existence}, and Theorem~\ref{convnames}. Then we define infinite {\it ranked} systems with asymmetric collisions. We prove an existence theorem: Theorem~\ref{thm2} for these systems. Unfortunately, we could not prove uniqueness: we just construct a copy of an infinite ranked system using approximation by finite ranked systems. This copy is called an {\it approximative version} of the infinite ranked system. We also develop comparison techniques for infinite systems, which parallel similar techniques for finite systems from \cite{MyOwn2}. 


Assume the usual setting: $(\Oa, \CF, (\CF_t)_{t \ge 0}, \MP)$, with the filtration satisfying the usual conditions. 

\subsection{Infinite classical systems} Fix parameters $g_1, g_2, \ldots \in \BR$ and $\si_1, \si_2, \ldots > 0$. We say that a sequence $(x_n)_{n \ge 1}$ of real numbers is {\it rankable} if there exists a one-to-one mapping (permutation) $\mP : \{1, 2, 3, \ldots\} \to \{1, 2, 3, \ldots\}$ which {\it ranks the components of $x$}: 
$$
x_{\mP(i)} \le x_{\mP(j)}\ \ \mbox{for}\ \ i, j = 1, 2, \ldots,\ \ i < j.
$$
As in the case of finite systems, we {\it resolve ties} (when $x_i = x_j$ for $i \ne j$) in the lexicographic order: we take a permutation $\mP$ which ranks the components of $x$, and, in addition, if $i < j$ and $x_{\mP(i)} = x_{\mP(j)}$, then $\mP(i) < \mP(j)$. There exists a unique such permutation $\mP$, which is called {\it the ranking permutation} and is denoted by $\mP_x$. For example, if $x = (2, 2, 1, 4, 5, 6, 7, \ldots)'$ (that is, $x(i) = i$ for $i \ge 4$), then $\mP_x(1) = 3,\ \mP_x(2) = 1,\ \mP_x(3) = 2,\ \mP_x(n) = n,\ n \ge 4$. Not all sequences of real numbers are rankable: for example, $x = (x_i = i^{-1},\, i \ge 1)$, is not rankable. 

\begin{defn} Consider an $\BR^{\infty}$-valued process 
$$
X = (X(t), t \ge 0),\ X(t) = (X_n(t))_{n \ge 1},
$$
with continuous adapted components, such that for every $t \ge 0$, the sequence $X(t) = (X_n(t))_{n \ge 1}$ is rankable. Let $\mP_t$ be the ranking permutation of $X(t)$. Let $W_1, W_2, \ldots$ be i.i.d. standard $(\CF_t)_{t \ge 0}$-Brownian motions. Assume that the process $X$ satisfies an SDE
$$
dX_i(t) = \SL_{k=1}^{\infty}1(\mP_t(k) = i)\left(g_k\md t + \si_k\md W_i(t)\right),\ \ i = 1, 2, \ldots
$$
Then the process $X$ is called an {\it infinite classical system of competing Brownian particles} with {\it drift coefficients} $(g_k)_{k \ge 1}$ and {\it diffusion coefficients} $(\si_k^2)_{k \ge 1}$. For each $i = 1, 2, \ldots$, the component $X_i = (X_i(t), t \ge 0)$ is called the {\it $i$th named particle}. If we define $Y_k(t) \equiv X_{\mP_t(k)}(t)$ for $t \ge 0$ and $k = 1, 2, \ldots$, then the process $Y_k = (Y_k(t), t \ge 0)$ is called the {\it $k$th ranked particle}. The $\BR^{\infty}_+$-valued process 
$$
Z = (Z(t), t \ge 0),\ \ Z(t) = (Z_k(t))_{k \ge 1},  
$$
defined by 
$$
Z_k(t) = Y_{k+1}(t) - Y_k(t),\ \ k = 1, 2, \ldots,\ t \ge 0,
$$
is called the {\it gap process}. If $X(0) = x \in \BR^{\infty}$, then we say that the system $X$ {\it starts from $x$}. This system is called {\it locally finite} if for any $u \in \BR$ and $T > 0$ there a.s. exists only finitely many $i \ge 1$ such that $\min_{[0, T]}X_i(t) \le u$. 
\label{defninfclassical}
\end{defn}

The following existence and uniqueness theorem was partially proved in \cite{IKS2013} and \cite{Shkolnikov2011}. We restate it here in a different form.

\begin{thm} Suppose $x \in \BR^{\infty}$ is a vector which satisfies the condition~\eqref{seriesalpha}. Assume also that there exists $n_0 \ge 1$ for which 
$$
g_{n_0+1} = g_{n_0+2} = \ldots\ \ \mbox{and}\ \ \si_{n_0+1} = \si_{n_0+2} = \ldots > 0.
$$
Then, in a weak sense there exists an infinite classical system of competing Brownian particles with drift coefficients $(g_k)_{k \ge 1}$ and diffusion coefficients $(\si_k^2)_{k \ge 1}$, starting from $x$, and it is unique in law. 
\label{existenceinfclassical}
\end{thm}

Let us also show a different existence and uniqueness result, analogous to \cite[Lemma 11]{PP2008}. 

\begin{thm}
\label{thm:Girsanov-existence}
Suppose $x \in \BR^{\infty}$ is a vector which satisfies the condition~\eqref{seriesalpha}. Assume also that 
$$
\si_n = 1,\ \ n \ge 1;\ \ \mbox{and}\ \ G := \SL_{n = 1}^{\infty}g_n^2 < \infty.
$$
Then in a weak sense there exists an infinite classical system of competing Brownian particles with drift coefficients $(g_k)_{k \ge 1}$ and diffusion coefficients $(\si_k^2)_{k \ge 1}$, starting from $x$, and it is unique in law. 
\end{thm}

Now, let us define an {\it approximative version} of an infinite classical system. 
Fix parameters $(g_n)_{n \ge 1}$ and $(\si_n^2)_{n \ge 1}$ and an initial condition $x = (x_i)_{i \ge 1}$. For each $N \ge 1$, consider a finite system of $N$ competing Brownian particles 
$$
X^{(N)} = \left(X^{(N)}_1, \ldots, X^{(N)}_N\right)'
$$
with drift coefficients $(g_n)_{1 \le n \le N}$ and diffusion coefficients $(\si_n^2)_{1 \le n \le N}$, starting from $[x]_N$. Let 
$$
Y^{(N)} = \left(Y^{(N)}_1, \ldots, Y^{(N)}_N\right)'
$$
be the ranked version of this system. Take an increasing sequence $(N_j)_{j \ge 1}$.

\begin{defn} Consider a version of the infinite classical system $X = (X_i)_{i \ge 1}$ of competing Brownian particles with parameters $(g_n)_{n \ge 1}$, $(\si_n^2)_{n \ge 1}$, starting from $x$. Let $Y_k$ be the $k$th ranked particle. Take an increasing sequence $(N_j)_{j \ge 1}$ of positive integers. Assume for every $T > 0$ and $M \ge 1$, weakly in $C([0, T], \BR^{2M})$, we have:
$$
\left(X^{(N_j)}_1, \ldots, X^{(N_j)}_M, Y^{(N_j)}_1, \ldots, Y^{(N_j)}_M\right)' 
\Ra \left(X_1, \ldots, X_M, Y_1, \ldots, Y_M\right)'.
$$
Then $X$ is called an {\it approximative version} of this infinite classical system, corresponding to the {\it approximation sequence} $(N_j)_{j \ge 1}$. 
\end{defn}

We prove weak existence (but not uniqueness in law) under the following conditions, which are slightly more general than the ones in Theorem~\ref{existenceinfclassical} and Theorem~\ref{thm:Girsanov-existence}. 

\begin{thm}
\label{convnames}
Consider parameters $(g_n)_{n \ge 1}$ and $(\si_n^2)_{n \ge 1}$ which satisfy 
\begin{equation}
\label{eq:bdd-parameters}
\ol{g} := \sup\limits_{n \ge 1}|g_n| < \infty,\ \ \mbox{and}\ \ \ol{\si}^2 := \sup\limits_{n \ge 1}\si_n^2 < \infty.
\end{equation}
Take initial conditions $x = (x_i)_{i \ge 1}$ satisfying the conditions~\eqref{seriesalpha}. 
Fix an increasing sequence $(N_j)_{j \ge 1}$. Then there exists a subsequence $(N'_j)_{j \ge 1}$ which serves as an approximation sequence for an approximative version $X$ of the infinite classical system of competing Brownian particles with parameters $(g_n)_{n \ge 1}$, $(\si_n^2)_{n \ge 1}$, starting from $x$. 
\end{thm}



This infinite classical system has the following properties. 

\begin{lemma} Consider any  infinite classical system $X = (X_i)_{i \ge 1}$, of competing Brownian particles with parameters $(g_n)_{n \ge 1}$, $(\si_n^2)_{n \ge 1}$, satisfying the condition~\eqref{eq:bdd-parameters}. Assume the initial condition $X(0) = x$ satisfies~\eqref{seriesalpha}. Then this system is locally finite. Also, the following set is the state space for $X = (X(t), t \ge 0)$:
$$
\mathcal V := \bigl\{x = (x_i)_{i \ge 1} \in \BR^{\infty}\mid \lim\limits_{i \to \infty}x_i = \infty\ \ \mbox{and}\ \ \SL_{i=1}^{\infty}e^{-\al x_i^2} < \infty\ \ \mbox{for all}\ \ \al > 0\bigr\}.
$$.
\label{names2ranks}
\end{lemma}

Now, let us describe the dynamics of the ranked particles $Y_k$. Denote by $L_{(k, k+1)}$ the local time process at zero of $Z_k$, $k = 1, 2, \ldots$ For notational convenience, let $L_{(0, 1)}(t) \equiv 0$. For $k = 1, 2, \ldots$ and $t \ge 0$, let 
$$
B_k(t) = \SL_{i=1}^{\infty}\int_0^t1(\mP_s(k) = i)\md W_i(s).
$$

\begin{lemma} Take a version of an infinite classical system of competing Brownian particles with parameters $(g_n)_{n \ge 1}$ and $(\si_n^2)_{n \ge 1}$. Assume this version is locally finite. Then the processes $B_k = (B_k(t), t \ge 0),\ k = 1, 2, \ldots$ are i.i.d. standard Brownian motions. For $t \ge 0$ and $k = 1, 2, \ldots$, we have:
\begin{equation}
\label{dynamics}
Y_k(t) = Y_k(0) + g_kt + \si_kB_k(t) - \frac12L_{(k, k+1)}(t) + \frac12L_{(k-1, k)}(t).
\end{equation}
\label{lemma:local-time}
\end{lemma}


\begin{lemma} Under conditions of Lemma~\ref{lemma:local-time}, for every $t > 0$ there is a.s. no tie at time $t > 0$. 
\label{lemma:notieclassical}
\end{lemma}

\subsection{Infinite systems with asymmetric collisions} Lemma~\ref{lemma:local-time} provides motivation to introduce infinite systems of competing Brownian particles with {\it asymmetric collisions}, when we have coefficients other than $1/2$ at the local times in~\eqref{dynamics}. We prove an existence theorem for these systems. Unfortunately, we could not prove uniqueness: we just construct a copy of an infinite ranked system using approximation by finite ranked systems. This copy is called the {\it approximative version} of the infinite ranked system. 

\begin{defn} Fix parameters $g_1, g_2, \ldots \in \BR$, $\si_1, \si_2, \ldots > 0$ and $(q^{\pm}_n)_{n \ge 1}$ such that
$$
q^+_{n+1} + q^-_n = 1,\ \ 0 < q^{\pm}_n < 1,\ \ n = 1, 2, \ldots
$$
Take a sequence of i.i.d. standard $(\CF_t)_{t \ge 0}$-Brownian motions $B_1, B_2, \ldots$ Consider an $\BR^{\infty}$-valued process $Y = (Y(t), t \ge 0)$ with continuous adapted components and continuous adapted real-valued processes $L_{(k, k+1)} = (L_{(k, k+1)}(t), t \ge 0),\ k = 1, 2, \ldots$ (for convenience, let $L_{(0, 1)} \equiv 0$), with the following properties:

\medskip

(i) $Y_1(t) \le Y_2(t) \le Y_3(t) \le \ldots$ for $t \ge 0$;

\medskip

(ii) for $k = 1, 2, \ldots,\ \ t \ge 0$, we have:
$$
Y_k(t) = Y_k(0) + g_kt + \si_kB_k(t) + q^+_kL_{(k-1, k)}(t) - q^-_kL_{(k, k+1)}(t);
$$

\medskip

(iii) each process $L_{(k, k+1)}$ is nondecreasing, $L_{(k, k+1)}(0) = 0$ and 
$$
\int_0^{\infty}\left(Y_{k+1}(t) - Y_k(t)\right)\md L_{(k, k+1)}(t) = 0,\ \ k = 1, 2, \ldots
$$
The last equation means that $L_{(k, k+1)}$ can increase only when $Y_k(t) = Y_{k+1}(t)$. 

\medskip

Then the process $Y$ is called an {\it infinite ranked system of competing Brownian particles} with {\it drift coefficients} $(g_k)_{k \ge 1}$, {\it diffusion coefficients} $(\si_k^2)_{k \ge 1}$, and {\it parameters of collisions} $(q^{\pm}_k)_{k \ge 1}$. The process $Y_k = (Y_k(t), t \ge 0)$ is called the {\it $k$th ranked particle}. The $\BR^{\infty}_+$-valued process $Z = (Z(t), t \ge 0),\ Z(t) = (Z_k(t))_{k \ge 1}$, defined by 
$$
Z_k(t) = Y_{k+1}(t) - Y_k(t),\ \ k = 1, 2, \ldots,\ \ t \ge 0,
$$
is called the {\it gap process}. The process $L_{(k, k+1)}$ is called the {\it local time of collision} between $Y_k$ and $Y_{k+1}$. If $Y(0) = y$, then we say that this system $Y$ {\it starts from} $y$. The processes $B_1, B_2, \ldots$ are called {\it driving Brownian motions}. The system $Y = (Y_k)_{k \ge 1}$ is called {\it locally finite} if for all $u \in \BR$ and $T > 0$ there exist only finitely many $k$ such that $\min_{[0, T]}Y_k(t) \le u$.
\label{main} 
\end{defn}

\begin{rmk}We can reformulate Lemma~\ref{lemma:local-time} as follows: take an infinite classical system $X = (X_i)_{i \ge 1}$ of competing Brownian particles with drift coefficients $(g_n)_{n \ge 1}$ and diffusion coefficients $(\si_n^2)_{n \ge 1}$. Rank this system $X$; in other words, switch from named particles $X_1, X_2, \ldots$, to ranked particles $Y_1, Y_2, \ldots$. The resulting system $Y = (Y_k)_{k \ge 1}$ is an infinite ranked system of competing Brownian particles with drift coefficients $(g_n)_{n \ge 1}$, diffusion coefficients $(\si_n^2)_{n \ge 1}$, and parameters of collision 
$\ol{q}^{\pm}_n = 1/2$, for $n \ge 1$. 
\end{rmk}

We construct this infinite system by approximating it with finite systems of competing Brownian particles with the same parameters. 

\begin{defn} Using the notation from Definition~\ref{main}, for every $N \ge 2$, let 
$$
Y^{(N)} = \left(Y^{(N)}_1, \ldots, Y^{(N)}_N\right)'
$$
be the system of $N$ ranked competing Brownian particles with drift coefficients $g_1, \ldots, g_N$, diffusion coefficients $\si_1^2, \ldots, \si_N^2$ and parameters of collision $(q^{\pm}_n)_{1 \le n \le N}$, driven by Brownian motions $B_1, \ldots, B_N$. Suppose there exist limits
$$
\lim\limits_{N \to \infty}Y^{(N)}_k(t) =: Y_k(t),
$$
which are uniform on every $[0, T]$, for every $k = 1, 2, \ldots$ Assume that $Y = (Y_k)_{k \ge 1}$ turns out to be an infinite system of competing Brownian particles with parameters $(g_n)_{n \ge 1}$, $(\si_n^2)_{n \ge 1}$, $(q^{\pm}_n)_{n \ge 1}$. Then we say that $Y$ is an {\it approximative version} of this system. 
\label{defapprox}
\end{defn}

\begin{rmk}
From Theorem~\ref{convnames}, Lemma~\ref{names2ranks}, and Lemma~\ref{lemma:local-time}, we know that if we take an approximative version of an infinite classical system of competing Brownian particles and rank it, we get the approximative version of an infinite ranked system. This allows us to use subsequent results of Sections 3, 4, and 5 for approximative versions of infinite classical systems. In particular, if (under conditions of Theorem~\ref{existenceinfclassical} or Theorem~\ref{thm:Girsanov-existence}) there is a unique in law version of an infinite classical system, then this only version is necessarily the approximative version, and we can apply results of Sections 3, 4, and 5 to this system. 
\end{rmk}

Now comes the main result of this subsection.

\begin{thm} 
\label{thm2} 
Take a sequence of drift coefficients $(g_n)_{n \ge 1}$, a sequence of diffusion coefficients $(\si_n^2)_{n \ge 1}$, and a sequence of parameters of collision $(q^{\pm}_n)_{n \ge 1}$. 
Suppose that the initial conditions $y \in \BR^{\infty}$ are such that 
$y_1 \le y_2 \le \ldots$, and 
$$
\SL_{n=1}^{\infty}e^{-\al y_n^2} < \infty\ \ \mbox{for all}\ \ \al > 0.
$$
Assume that 
\begin{equation}
\label{eq:bounds-g-si}
\inf\limits_{n \ge 1}g_n =: \underline{g} > -\infty,\ \ \sup\limits_{n \ge 1}\si_n^2  =: \ol{\si}^2 < \infty,
\end{equation}
and there exists $n_0 \ge 1$ such that 
\begin{equation}
\label{eq:ge12}
q^+_n \ge \frac12\ \ \mbox{for}\ \ n \ge n_0.
\end{equation}
Take any i.i.d. standard Brownian motions $B_1, B_2, \ldots$ Then there exists the approximative version of the infinite ranked system of competing Brownian particles with parameters 
$$
(g_n)_{n \ge 1},\ (\si_n^2)_{n \ge 1},\ (q^{\pm}_n)_{n \ge 1},
$$
starting from $y$, with driving Brownian motions $B_1, B_2, \ldots$
\end{thm}

\begin{rmk} We have not proved uniqueness for infinite ranked system from Theorem~\ref{thm2}. We can so far only claim uniqueness for infinite {\it classical} systems. Now, suppose we take the infinite ranked system from Theorem~\ref{thm2} with symmetric collisions, when $q^{\pm}_n = 1/2$ for all $n$. Under the additional assumption that  this system must be the result of ranking a classical system, we also get uniqueness. But without this special condition, it is not known whether this ranked system is unique. 
\end{rmk}

Let us now present some additional properties of this newly constructed approximative version of an infinite system of competing Brownian particles. These are analogous to the properties of an infinite classical system of competing Brownian particles, stated in Lemma~\ref{names2ranks} and Lemma~\ref{lemma:notieclassical} above. 

\begin{lemma} An approximative version of an infinite ranked system from Theorem~\ref{thm2} is locally finite. The process $Y = (Y(t), t \ge 0)$ has the state space 
$$
\mathcal W := \bigl\{y = (y_k)_{k \ge 1} \in \BR^{\infty}\mid y_1 \le y_2 \le y_3 \le \ldots,\ \lim\limits_{k \to \infty}y_k = \infty,\  \SL_{k=1}^{\infty}e^{-\al y_k^2} < \infty,\  \mbox{for all}\  \al > 0\bigr\}.
$$
\label{aux}
\end{lemma}

\begin{lemma} Consider an infinite system from Definition~\ref{main}, which is locally finite. Then for every $t > 0$ a.s. the vector $Y(t) = (Y_k(t))_{k \ge 1}$ has no ties. 
\label{notieinf}
\end{lemma}

\subsection{Comparison techniques for infinite systems} We developed comparison techniques for finite systems of competing Brownian particles in \cite{MyOwn2}. These techniques also work for approximative versions of infinite ranked systems. By taking limits as the number $N$ of particles goes to infinity, we can formulate the same comparison results for these two infinite systems. Let us give a few examples. The proofs trivially follow from the corresponding results for finite systems from \cite[Section 3]{MyOwn2}. These techniques are used later in Section 4 of this article, as well as in proofs of statements from Section 3. 

\begin{cor} Take two approximative versions $Y$ and $\ol{Y}$ of an infinite system of competing Brownian particles with the same parameters  
$$
(g_n)_{n \ge 1},\ (\si_n^2)_{n \ge 1},\ (q^{\pm}_n)_{n \ge 1},
$$
with the same driving Brownian motions, but starting from different initial conditions $Y(0)$ and $\ol{Y}(0)$. Let $Z$ and $\ol{Z}$ be the corresponding gap processes, and let $L$ and $\ol{L}$ be the corresponding vectors of local time terms. Then the following inequalities hold a.s.:
\label{comparinfcor}

\smallskip

(i) If $Y(0) \le \ol{Y}(0)$, then $Y(t) \le \ol{Y}(t)$, $t \ge 0$. 

\smallskip

(ii) If $Z(0) \le \ol{Z}(0)$, then $Z(t) \le \ol{Z}(t)$, $t \ge 0$, and $L(t) - L(s) \ge \ol{L}(t) - \ol{L}(s)$, $0 \le s \le t$. 
\end{cor}

\begin{cor} Fix $M \ge 2$. Take two approximative versions $Y = (Y_n)_{n \ge M}$ and $\ol{Y} = (\ol{Y}_n)_{n \ge 1}$ of an infinite system of competing Brownian particles with parameters 
$$
(g_n)_{n \ge M},\ (\si_n^2)_{n \ge M},\ (q^{\pm}_n)_{n \ge M};
$$
$$
(g_n)_{n \ge 1},\ (\si_n^2)_{n \ge 1},\ (q^{\pm}_n)_{n \ge 1}.
$$
Assume that $Y_k(0) = \ol{Y}_k(0)$ for $k \ge M$. If $B_1, B_2, \ldots$ are driving Brownian motions for $\ol{Y}$, then let $B_M, B_{M+1}, \ldots$ be the driving Brownian motions for $Y$. Let $Z = (Z_k)_{k \ge M}$ and $\ol{Z} = (\ol{Z}_k)_{k \ge 1}$ be the corresponding gap processes, and let $L = (L_{(k, k+1)})_{k \ge M}$ and $\ol{L} = (\ol{L}_{(k, k+1)})_{k \ge 1}$ be the vectors of boundary terms. Then a.s. the following inequalities hold:
$$
Y_k(t) \le \ol{Y}_k(t),\ k \ge M,\ t \ge 0;
$$
$$
L_{(k, k+1)}(t) - L_{(k, k+1)}(s) \le \ol{L}_{(k, k+1)}(t) - \ol{L}_{(k, k+1)}(s),\ \ 0 \le s \le t,\ \ k \ge M;
$$
$$
Z_k(t) \ge \ol{Z}_k(t),\ \ t \ge 0,\ \ k \ge M. 
$$
\label{cor:comp2}
\end{cor}

\begin{cor} Take two approximative versions $Y$ and $\ol{Y}$ of an infinite system of competing Brownian particles with parameters 
$$
(g_n)_{n \ge 1},\ (\si_n^2)_{n \ge 1},\ (q^{\pm}_n)_{n \ge 1};
$$
$$
(\ol{g}_n)_{n \ge 1},\ (\si_n^2)_{n \ge 1},\ (\ol{q}^{\pm}_n)_{n \ge 1},
$$
with the same driving Brownian motions, starting from the same initial conditions. Let $Z$ and $\ol{Z}$ be the corresponding gap processes. Then:
\label{comparinfcordrifts}

\smallskip

(i) If $q^{\pm}_n = \ol{q}^{\pm}_n$, but $g_n \le \ol{g}_n$ for $n = 1, 2, \ldots$, then $Y(t) \le \ol{Y}(t)$, $t \ge 0$;

\smallskip

(ii) If $q^{\pm}_n = \ol{q}^{\pm}_n$, but $g_{n+1} - g_n \le \ol{g}_{n+1} - \ol{g}_n$ for $n = 1, 2, \ldots$, then  $Z(t) \le \ol{Z}(t)$, $t \ge 0$;

\smallskip

(iii) If $g_n = \ol{g}_n$, but $q^+_n \le \ol{q}^+_n$ for $n = 1, 2, \ldots$, then $Y(t) \le \ol{Y}(t)$, $t \ge 0$. 
\end{cor}

\begin{rmk} Suppose that in each of these three corollaries, we remove the requirement that the two infinite systems have the same driving Brownian motions. Then we get stochastic ordering instead of pathwise ordering. The same applies to Corollary~\ref{comparinfcor} if we switch from a.s. comparison to stochastic comparison in the inequalities $Y(0) \le \ol{Y}(0)$ and $Z(0) \le \ol{Z}(0)$, respectively. 
\label{RemarkA}
\end{rmk}

\section{The Gap Process: Stationary Distributions and Weak Convergence}

In this section, we construct a stationary distribution $\pi$ for the gap process $Z = (Z(t), t \ge 0)$ of such system. Then we prove weak convergence results for $Z(t)$ as $t \to \infty$.  

\subsection{Construction of a stationary distribution}

Consider again an infinite system $Y$ of competing Brownian particles with parameters $(g_n)_{n \ge 1}$, $(\si_n^2)_{n \ge 1}$, $(q^{\pm}_n)_{n \ge 1}$. Let $Z$ be its gap process. Let us recall a definition from the Introduction. 

\begin{defn} Let $\pi$ be a probability measure on $\BR^{\infty}_+$. We say that $\pi$ {\it is a stationary distribution} for the gap process for the system above if there exists a version $Y$ of this system such that for every $t \ge 0$, we have: $Z(t) \sim \pi$. 
\end{defn}

Let us emphasize that in this section, we do not study uniqueness and Markov property. We simply construct a copy of the system with required properties.  


\begin{asmp} Consider, for each $N \ge 2$, the ranked system of $N$ competing Brownian particles with parameters $(g_n)_{1 \le n \le N}, (\si_n^2)_{1 \le n \le N}, (q^{\pm}_n)_{1 \le n \le N}$. There exists a sequence $(N_j)_{j \ge 1}$ such that $N_j \to \infty$ and for every $j \ge 1$, the system of $N = N_j$ particles is such that its gap process has a stationary distribution. Let $\pi^{(N_j)}$ be this stationary distribution on $\BR^{N_j-1}_+$.
\end{asmp}

Define an $(N-1)\times(N-1)$-matrix $R^{(N)}$ and a vector $\mu^{(N)}$ from $\BR^{N-1}$, as in~\eqref{R} and~\eqref{mu}. By Proposition~\ref{basic}, Assumption 1 holds if and only if
$$
[R^{(N_j)}]^{-1}\mu^{(N_j)} < 0.
$$

Let $B_1, B_2, \ldots$ be i.i.d. standard Brownian motions. Let $z^{(N_j)} \sim \pi^{(N_j)}$ be an $\CF_0$-measurable random variable. Consider the system 
$\ol{Y}^{(N_j)}$ of $N_j$ ranked competing Brownian particles with parameters 
$$
(g_n)_{1 \le n \le N_j}, (\si_n^2)_{1 \le n \le N_j}, (q^{\pm}_n)_{1 \le n \le N_j},
$$
starting from 
$$
(0, z^{(N_j)}_1, \ldots, z_1^{(N_j)} + \ldots + z_{N_j-1}^{(N_j)})',
$$
with driving Brownian motions $B_1, \ldots, B_{N_j}$. The following statement, which we state separately as a lemma, is a direct corollary of \cite[Corollary 3.14]{MyOwn2}. 

\begin{lemma} $[\pi^{(N_{j+1})}]_{N_j-1} \preceq \pi^{(N_j)}$. 
\label{Lemma123}
\end{lemma}

Without loss of generality, by changing the probability space we can take $z^{(N_j)} \sim \pi^{(N_j)}$ such that a.s. 
$[z^{(N_{j+1})}]_{N_j-1} \le z^{(N_j)}$, for $j \ge 1$. In other words, for all $j = 1, 2, \ldots$ and $k = 1, \ldots, N_j-1$, we have:
$$
0 \le z^{(N_{j+1})}_k \le z_k^{(N_j)}.
$$
A bounded monotone sequence has a limit:
$$
z_k = \lim\limits_{j \to \infty}z^{(N_j)}_k,\ \ k \ge 1. 
$$
Denote by $\pi$ the distribution of $(z_1, z_2, \ldots)$ on $\BR^{\infty}_+$. Then $\pi$ becomes a prospective stationary distribution for the gap process for the infinite system of competing Brownian particles. Equivalently, we can define $\pi$ as follows: for every $M \ge 1$, let 
$$
[\pi^{(N_j)}]_{M}\ \Ra\ \rho^{(M)},\ \ j \to \infty.
$$
These finite-dimensional distributions $\rho^{(M)}$ are {\it consistent}: 
$$
[\rho^{(M+1)}]_{M} = \rho^{(M)},\ \ M \ge 1.
$$
By Kolmogorov's theorem there exists a unique distribution $\pi$ on $\BR^{\infty}_+$ such that $[\pi]_{M} = \rho^{(M)}$ for all $M \ge 1$. Note that this limiting distribution does not depend on the sequence $(N_j)_{j \ge 1}$, as shown in the next lemma.

\begin{lemma} If there exist two sequences $(N_j)_{j \ge 1}$ and $(\tilde{N}_j)_{j \ge 1}$ which satisfy Assumption 1, and if $\pi$ and $\tilde{\pi}$ are the resulting limiting distributions, then $\pi = \tilde{\pi}$.
\label{lemma:independent-sequence} 
\end{lemma}

The next lemma allows us to rewrite the condition~\eqref{seriesalpha} in terms of the gap process. 

\begin{lemma} For a sequence $y = (y_n)_{n \ge 1} \in \BR^{\infty}$ such that $y_n \le y_{n+1},\ n \ge 1$, let $z = (z_n)_{n \ge 1} \in \BR^{\infty}$ be defined by $z_n = y_{n+1} - y_n,\ n \ge 1$. Then $y$ satisfies~\eqref{seriesalpha} if and only if $z$ satisfies
\begin{equation}
\label{gapnice}
\SL_{n=1}^{\infty}\exp\left(-\al(z_1 + \ldots + z_n)^2\right) < \infty\ \ \mbox{for all}\ \ \al > 0.
\end{equation}  
\label{aux:gaps}
\end{lemma}

Now, let us state one of the two main results of this section. 

\begin{thm} Consider an infinite system of competing Brownian particles with parameters 
$$
(g_n)_{n \ge 1},\ (\si_n^2)_{n \ge 1},\ (q^{\pm}_n)_{n \ge 1}.
$$

(i) Let the Assumption 1 and~\eqref{eq:bounds-g-si},~\eqref{eq:ge12} be true. Then we can construct the distribution $\pi$. 

(ii) Assume, in addition, that if a $\BR^{\infty}_+$-valued random variable $z$ is distributed according to $\pi$, then $z = (z_1, z_2,\ldots)$ a.s. satisfies~\eqref{gapnice}. Then we can construct an approximative version of the infinite system of competing Brownian particles with parameters 
$$
(g_n)_{n \ge 1},\ (\si_n^2)_{n \ge 1},\ (q^{\pm}_n)_{n \ge 1},
$$
such that $\pi$ is a stationary distribution for the gap process. 
\label{thm3}
\end{thm}

\begin{rmk} As mentioned in the Introduction, if a stationary distribution for the gap process of {\it finite} systems exists, it is unique. This was proved in \cite{5people}. For infinite systems, this is an open question.  
\end{rmk} 

In this subsection, we apply Theorem~\ref{thm3} to the case of the {\it skew-symmetry condition}, similar to~\eqref{SS}: 
\begin{equation}
\label{SSinf}
(q^-_{k-1} + q^+_{k+1})\si_k^2 = q^-_k\si^2_{k+1}  + q^+_k\si_{k-1}^2,\ \ k = 2, 3, \ldots
\end{equation}
Under this condition, by Proposition~\ref{basic}, 
$$
\pi^{(N_j)} = \bigotimes\limits_{k=1}^{N_j-1}\Exp(\la^{(N_j)}_k),
$$
where we define for $k = 1, \ldots, N_j-1$:
$$
\la_k^{(N_j)} = \frac2{\si_k^2 + \si_{k+1}^2}\left(-[R^{(N_j)}]^{-1}\mu^{(N_j)}\right)_k.
$$
Consider the following marginal of this stationary distribution:
$$
[\pi^{(N_{j+1})}]_{N_j-1} = \bigotimes\limits_{k=1}^{N_j-1}\Exp(\la^{(N_{j+1})}_k).
$$
By Lemma~\ref{Lemma123}, we can compare:
$$
[\pi^{(N_{j+1})}]_{N_j-1} \preceq \pi^{(N_j)} = \bigotimes\limits_{k=1}^{N_j-1}\Exp(\la^{(N_j)}_k).
$$
But $\Exp(\la') \preceq \Exp(\la'')$ is equivalent to $\la' \ge \la''$. Therefore, $\la^{(N_j)}_k \le \la^{(N_{j+1})}_k$, for $k = 1, \ldots, N_j-1$. In other words, for every $k$, the sequence $(\la^{(N_j)}_k)$ is nondecreasing. There exists a limit (possibly infinite) 
$$
\la_k := \lim\limits_{j \to \infty}\la^{(N_j)}_k,\ \ k = 1, 2, \ldots
$$
Assume that $\la_k < \infty$ for all $k = 1, 2, \ldots$ Then 
\begin{equation}
\label{cand}
\pi = \bigotimes\limits_{k=1}^{\infty}\Exp(\la_k).
\end{equation}
If $\la_k = \infty$ for some $k$, then we can also write~\eqref{cand}, understanding that $\Exp(\infty) = \de_{0}$ is the Dirac point mass at zero. This $\pi$ is a candidate for a stationary distribution. If the condition~\eqref{gapnice} is satisfied $\pi$-a.s., then $\pi$ is, indeed, a stationary distribution. Let us give a sufficient condition for~\eqref{gapnice}. 

\begin{lemma} Consider a distribution $\pi$ as in~\eqref{cand}. Let $\La_n := \sum_{k=1}^n\la_k^{-1}$. 

\medskip

(i) If $\sup_{n \ge 1}\la_n  < \infty$, then $\pi$-a.s.~\eqref{gapnice} is satisfied.

%
%

(ii) If $\sum_{n = 1}^{\infty}\la_n^{-2} < \infty$, then $\pi$-a.s.~\eqref{gapnice} is satisfied if and only if  
\begin{equation}
\label{La}
\SL_{n=1}^{\infty}e^{-\al \La_n^2} < \infty\ \ \mbox{for all}\ \ \al > 0.
\end{equation}
\label{satisfaction}
\end{lemma}


\subsection{The case of symmetric collisions} 
Assume now that the collisions are symmetric: $q^{\pm}_n = 1/2$, $n = 1, 2, \ldots
$ Then the skew-symmetry condition~\eqref{SSinf} takes the form $\si_{k+1}^2 - \si_k^2 = \si_k^2 - \si_{k-1}^2$, for $k = 2, 3, \ldots$. In other words, $\si_k^2$ must linearly depend on $k$. If, in addition,~\eqref{eq:bounds-g-si} holds, then $\si_k^2 = \si^2,\ k = 1, 2, \ldots$ Recall the definition of $\ol{g}_k$ from~\eqref{eq:stability}. 
 It was shown in Proposition~\ref{basicsymm} that in this case, $[R^{(N_j)}]^{-1}\mu^{(N_j)} < 0$ if and only if 
\begin{equation}
\label{eq:comp-of-g-bar}
\ol{g}_k > \ol{g}_{N_j},\ k = 1, \ldots, N_j-1.
\end{equation}
If the inequality~\eqref{eq:comp-of-g-bar} is true for $j = 1, 2, \ldots$, then 
$$
\pi^{(N_j)} = \bigotimes\limits_{k=1}^{N_j-1}\Exp\left(\la^{(N_j)}_k\right),\ \ \la^{(N_j)}_k := \frac{2k}{\si^2}\left(\ol{g}_k - \ol{g}_{N_j}\right).
$$
Assume the sequence $(g_n)_{n \ge 1}$ is bounded from below, as in~\eqref{eq:bounds-g-si}. Then the sequence $(\ol{g}_{N_j})_{j \ge 1}$ is also bounded below. From~\eqref{eq:comp-of-g-bar}, we get: $\ol{g}_{N_j} > \ol{g}_{N_{j+1}}$ for $j = 1, 2, \ldots$. Therefore, there exists the limit $\lim_{j \to \infty}\ol{g}_{N_j} =: \ol{g}_{\infty}$. Then, as $j \to \infty$, we get: 
$$
\la^{(N_j)}_k \to \la_k := \frac{2k}{\si^2}\left(\ol{g}_k - \ol{g}_{\infty}\right).
$$
Thus, the distribution $\pi$ has the following product-of-exponentials form:
\begin{equation}
\label{eq:pi-SS-symm}
\pi = \bigotimes\limits_{k=1}^{\infty}\Exp(\la_k) = \bigotimes\limits_{k=1}^{\infty}\Exp\left(\frac{2k}{\si^2}\left(\ol{g}_k - \ol{g}_{\infty}\right)\right).
\end{equation}
If $\la_k,\ k = 1, 2, \ldots$, satisfy Lemma~\ref{satisfaction}, then $\pi$ is a stationary distribution. 

\begin{exm} Consider an infinite system with symmetric collisions, with drift and diffusion coefficients
$$
g_1, g_2, \ldots, g_M > 0,\ g_{M+1} = g_{M+2} = \ldots = 0,\ \si_1 = \si_2 = \ldots = 1.
$$
Then 
$$
\ol{g}_k = \frac{g_1 + \ldots + g_M}k,\ k > M.
$$
Therefore, $\ol{g}_{\infty} = \lim_{k \to \infty}\ol{g}_k = 0$, and the parameters $\la_k$ from~\eqref{eq:pi-SS-symm} are equal to 
$$
\la_k = \begin{cases}
2(g_1+\ldots + g_k),\ 1 \le k \le M;\\
2(g_1 + \ldots + g_M),\ k > M.
\end{cases}
$$
These parameters satisfy  Lemma~\ref{satisfaction} (i). Therefore, the conclusions of this section are valid. In particular, if $g_1 = \ldots = g_M = 1$, as in Theorem~\ref{thm1} from the Introduction, then 
$$
\pi = \Exp(2)\otimes\Exp(4)\otimes\ldots\otimes\Exp(2M)\otimes\Exp(2M)\otimes\ldots
$$
\end{exm}

\subsection{Convergence Results}

Now, consider questions of {\it convergence} of the gap process as $t \to \infty$ to the stationary distribution $\pi$ constructed above. Let us outline the facts proved in this subsection (omitting the required conditions for now). 

\medskip

(a) The family of random variables $Z(t), t \ge 0$, is tight in $\BR^{\infty}_+$ with respect to the componentwise convergence (which is metrizable by a certain metric). Any weak limit point of $Z(t)$ as $t \to \infty$ is stochastically dominated by $\pi$. 

\medskip

(b) If we start the approximative version of the infinite system $Y$ with gaps stochastically larger than $\pi$, then the gap process converges weakly to $\pi$. 

\medskip

(c) Any other stationary distribution for the gap process (if it exists) {\it must} be stochastically smaller than $\pi$. 

\medskip

The rest of this subsection is devoted to precise statements of these results.

\begin{thm} Consider any version (not necessarily approximative) of the infinite system of competing Brownian particles with parameters 
$$
(g_n)_{n \ge 1},\ (\si_n^2)_{n \ge 1},\ (q^{\pm}_n)_{n \ge 1}.
$$
Suppose Assumption 1 holds. 

\medskip

(i) Then the family of $\BR^{\infty}_+$-valued random variables $Z(t),\ t \ge 0$ is tight in $\BR^{\infty}_+$. 

\medskip

(ii) Suppose for some sequence $t_j \uparrow \infty$ we have: $Z(t_j)\ \Ra\ \nu$ as $j \to \infty$, where $\nu$ is some probability measure on $\BR^{\infty}_+$. Then $\nu \preceq \pi$: the measure $\nu$ is stochastically dominated by $\pi$. 

\medskip

(iii) Under conditions of Theorem~\ref{thm3} (ii), every stationary distribution $\pi'$ for the gap process is stochastically dominated by $\pi$: 
$\pi' \preceq \pi$.
\label{limitpointthm}
\end{thm}

\begin{rmk} Let us stress: we do not need $Y$ to be an approximative version of the system, and we do not need the initial conditions $Y(0) = y$ to satisfy~\eqref{seriesalpha}. 
\end{rmk}

\begin{thm} Consider an approximative version $Y$ of the infinite system of competing Brownian particles with parameters $(g_n)_{n \ge 1},\ (\si_n^2)_{n \ge 1},\ (q^{\pm}_n)_{n \ge 1}$. Let $Z$ be the corresponding gap process. Suppose it satisfies conditions of Theorem~\ref{thm3} (ii). Then we can construct the distribution $\pi$, and it is a stationary distribution for the gap process. If $Z(0) \succeq \pi$, then 
$$
Z(t) \Ra \pi,\ \ t \to \infty.
$$
\label{thm:main-conv}
\end{thm}

\begin{proof} Let us show that for each $t \ge 0$ we have: $Z(t) \succeq \pi$. (Together with Theorem~\ref{limitpointthm} (i), (ii), this completes the proof.) Consider another system $\ol{Y}$: an approximative version of the system with the gap process $\ol{Z}$ having stationary distribution $\pi$. Then $Z(0) \succeq \ol{Z}(0) \sim \pi$. 
By Corollary~\ref{comparinfcor} (ii) above, $Z(t) \succeq \ol{Z}(t) \sim \pi,\ t \ge 0$. 
\end{proof}

\section{Triple Collisions for Infinite Systems} 

Let us define triple and simultaneous collisions for an infinite ranked system $Y = (Y_n)_{n \ge 1}$ of competing Brownian particles. 

\begin{defn} We say that a {\it triple collision between particles $Y_{k-1}$, $Y_k$ and $Y_{k+1}$} occurs at time $t \ge 0$ if 
$$
Y_{k-1}(t) = Y_k(t) = Y_{k+1}(t).
$$
We say that a {\it simultaneous collision} occurs at time $t \ge 0$ if for some $1 \le k < l$, we have:
$$
Y_k(t) = Y_{k+1}(t)\ \ \mbox{and}\ \ Y_l(t) = Y_{l+1}(t).
$$
\end{defn}

A triple collision is a particular case of a simultaneous collision. For finite systems of competing Brownian particles (both classical and ranked), the question of a.s. absence of triple collisions was studied in \cite{IK2010, IKS2013, KPS2012}. A necessary and sufficient condition for a.s. absence of any triple collisions was found in \cite{MyOwn3}; see also \cite{MyOwn5} for related work. 
This condition also happens to be sufficient for a.s. absence of any simultaneous collisions. In general, triple collisions are undesirable, because strong existence and pathwise uniqueness for classical systems of competing Brownian particles was shown in \cite{IKS2013} only up to the first moment of a triple collision. Some results about triple collisions for infinite classical systems were obtained in the paper \cite{IKS2013}. Here, we strengthen them a bit and also prove results for asymmetric collisions. 

It turns out that the same necessary and sufficient condition works for infinite systems as well as for finite systems. 

\begin{thm} Consider a version of the infinite ranked system of competing Brownian particles $Y = (Y_n)_{n \ge 1}$ with parameters
$$
(g_n)_{n \ge 1},\ (\si_n^2)_{n \ge 1},\ (q^{\pm}_n)_{n \ge 1}.
$$

\medskip

(i) Assume this version is locally finite. If
\begin{equation}
\label{conditionasymm:inf}
(q^-_{k-1} + q^+_{k+1})\si_k^2 \ge q^-_k\si^2_{k+1}  + q^+_k\si_{k-1}^2,\ \ k = 2, 3, \ldots
\end{equation}
Then a.s. for any $t > 0$ there are no triple and no simultaneous collisions at time $t$. 

\medskip

(ii) If the condition~\eqref{conditionasymm:inf} is violated for some $k = 2, 3, \ldots$, then with positive probability there exists a moment $t > 0$ such that there is a triple collision between particles with ranks $k-1,\ k$, and $k+1$ at time $t$. 
\label{thmasymminf}
\end{thm}

An interesting corollary of \cite[Theorem 1.2]{MyOwn3} for finite systems is that if there are a.s. no triple collisions, then there are also a.s. no simultaneous collisions. This is also true for infinite systems constructed in Theorem~\ref{thm2}. 

\begin{rmk} For symmetric collisions: $q^{\pm}_n = 1/2,\ n = 1, 2, \ldots$, this result takes the following form. 
There are a.s. no triple collisions if and only if the sequence $(\si_k^2)_{k \ge 1}$ is concave. In this case, there are also a.s. no simultaneous collisions. If for some $k \ge 1$ we have:
$$
\si_{k+1}^2 < \frac12\left(\si_k^2 + \si_{k+2}^2\right),
$$
then with positive probability there exists $t > 0$ such that $Y_k(t) = Y_{k+1}(t) = Y_{k+2}(t)$. 
\label{rmk:symm-3ple}
\end{rmk}

\begin{rmk} Let us restate the main result of \cite{IKS2013}: for a infinite classical systems of competing Brownian particles which satisfies conditions of Theorem~\ref{existenceinfclassical}, there exists a unique strong version up to the first triple collision. In particular, if the sequence of diffusion coefficients $(\si_k^2)_{k \ge 1}$ is concave, then there exists a unique strong solution on the infinite time horizon. 
\end{rmk}

\begin{rmk} Partial results of \cite{IKS2013} for infinite classical systems of competing Brownian particles are worth mentioning: if there are a.s. no triple collisions, then $(\si_k^2)_{k \ge 1}$ is concave; if the sequence $(0, \si_1^2, \si_2^2, \ldots)$ is concave, then there are a.s. no triple collisions. In particular, it was already shown in \cite{IKS2013} that the model~\eqref{eq:M-Atlas}, as any model with $\si_1 = \si_2 = \ldots = 1$, a.s. does not have triple collisions. 
\end{rmk}

\section{Proofs}

\subsection{Proof of Proposition~\ref{basic}} The concept of a {\it semimartingale reflected Brownian motion} (SRBM) in the positive orthant $\BR^d_+$ is discussed in the survey \cite{WilliamsSurvey}; we refer the reader to this article for definition and main known results about this process. Here, we informally introduce the concept. Take a $d\times d$-matrix $R$ with diagonal elements equal to $1$, and denote by $r_i$ the $i$th column of $R$. Next, take a symmetric positive definite $d\times d$-matrix $A$, as well as $\mu \in \BR^d$. A {\it semimartingale reflected Brownian motion} (SRBM) {\it in the orthant} with {\it drift vector} $\mu$, {\it covariance matrix} $A$, and {\it reflection matrix} $R$ is a Markov process in $\BR^d_+$ such that:

\medskip

(i) when it is in the interior of the orthant, it behaves as a $d$-dimensional Brownian motion with drift vector $\mu$ and covariance matrix $A$;

\medskip

(ii) at each face $\{x \in \BR^d_+\mid x_i = 0\}$ of the boundary of this orthant, it is reflected instantaneously according to the vector $r_i$ (if $r_i = e_i$, which is the $i$th standard unit vector in $\BR^d$, this is {\it normal} reflection). 

\medskip

It turns out that $Z$ is an SRBM in the orthant $\BR^{N-1}_+$ with reflection matrix $R$ given by~\eqref{R}, drift vector $\mu$ as in~\eqref{mu}, and covariance matrix 
\begin{equation}
\label{A}
A = 
\begin{bmatrix}
\si_1^2 + \si_2^2 & -\si_2^2 & 0 & 0 & \ldots & 0 & 0\\
-\si_2^2 & \si_2^2 + \si_3^2 & -\si_3^2 & 0 & \ldots & 0 & 0\\
0 & -\si_3^2 & \si_3^2 + \si_4^2 & -\si_4^2 & \ldots & 0 & 0\\
\vdots & \vdots & \vdots & \vdots & \ddots & \vdots & \vdots\\
0 & 0 & 0 & 0 & \ldots & \si_{N-2}^2 + \si_{N-1}^2 & -\si_{N-1}^2\\
0 & 0 & 0 & 0 & \ldots & -\si_{N-1}^2 & \si_{N-1}^2 + \si_N^2
\end{bmatrix}
\end{equation}
See \cite[subsection 2.1]{KPS2012}, \cite{MyOwn3, 5people}. The results of Proposition~\ref{basic} follow from the properties of an SRBM. Property (i) of the matrix $R$ was proved in \cite[subsection 2.1]{KPS2012}; see also \cite[Lemma 2.9]{MyOwn3}. The {\it skew-symmetry condition} for an SRBM is written in the form
$$
RD + DR' = 2A,
$$
where $D = \diag(A)$ is the $(N-1)\times(N-1)$-diagonal matrix with the same diagonal entries as $A$. As mentioned in \cite[Theorem 3.5]{WilliamsSurvey}, this is a necessary and sufficient condition for the stationary distribution to have product-of-exponentials form. This condition can be rewritten for $R$ and $A$ from~\eqref{R} and~\eqref{A} as~\eqref{SS}. 

%

\subsection{Proof of Lemma~\ref{notie}} There is a tie for a system of competing Brownian particles at time $t > 0$ if and only if the gap process at time $t$ hits the boundary of the orthant $\BR^{N-1}_+$. But the gap process is an SRBM $Z = (Z(t), t \ge 0)$ in $\BR^{N-1}_+$, with the property from \cite{RW1988}: $\MP(Z(t) \in \pa\BR^{N-1}_+) = 0$ for every $t > 0$. 

\subsection{Proof of Theorem~\ref{existenceinfclassical}} Because of the results of \cite{IKS2013}, we need only to prove the following condition. Fix  $T > 0$ and $x \in \BR$. Let $\Xi$ be the set of all progressively measurable real-valued processes $\zeta = (\zeta(t))_{0 \le t \le T}$ with values in $[\min_{i \ge 1}\si_i, \max_{i \ge 1}\si_i]$. Then for every $\zeta \in \Xi$, 
\begin{equation}
\label{seriescomplex}
\SL_{i=1}^{\infty}\sup\limits_{\xi \in \Xi}\MP\left(x_i - \ol{g}T - \max\limits_{0 \le t \le T}\int_0^t\zeta(s)\md W_i(s) < x\right) < \infty,
\end{equation}
But this follows from Lemma~\ref{lemma:estimate-special} and Lemma~\ref{lemma:special-series}.

\subsection{Proof of Theorem~\ref{thm:Girsanov-existence}} The proof closely follows that of \cite[Lemma 11]{PP2008}. Assume without loss of generality that initially, the particles are ranked, that is, $x_k \le x_{k+1}$ for $k \ge 1$. Consider i.i.d. standard Brownian motions $W_1, W_2, \ldots$, and let $X_i(t) = x_i + W_i(t),\ i \ge 1$. 

\begin{lemma} For every $t \ge 0$, the sequence $X(t) = (X_n(t))_{n \ge 1}$ is rankable.
\end{lemma}

\begin{proof} It suffices to show that the system $X$ is locally finite. This statement follows from Lemmata~\ref{lemma:estimate-special},~\ref{lemma:special-series}, the Borel-Cantelli lemma and the fact that the initial condition $x$ satisfies~\eqref{seriesalpha}. 
\end{proof}

Recall our standard setting: $(\Oa, \CF, (\CF_t)_{t \ge 0}, \MP)$. Let $\mP_t$ be the ranking permutation of the sequence $X(t)$. Fix $T > 0$ and apply Girsanov theorem to $X = (X_n)_{n \ge 1}$ on $\CF_T$. 
We construct the new measure
$$
\left.\mathbf Q\right|_{\CF_t} = D(t)\cdot\left.\MP\right|_{\CF_t},\ \ \mbox{where}\ \ 
D(t) := \exp\Bigl(M_{\infty}(t) - \frac12\langle M_{\infty}\rangle_t\Bigr),\ \ t \ge 0,
$$
and 
\begin{equation}
\label{eq:M-infty}
M_{\infty}(t) = \SL_{i=1}^{\infty}\SL_{k=1}^{\infty}\int_0^tg_k1\left(\mP_s(k) = i\right)\md W_i(s).
\end{equation}
It suffices to show that the process $M_{\infty}$ exists and is a continuous square-integrable martingale, with $\langle M_{\infty}\rangle_t = Gt$ for all $t \ge 0$. 
Indeed, the rest follows from Girsanov theorem. Fix $T > 0$. Consider the space $\mathbb M$ of continuous square-integrable martingales $M = (M(t), 0 \le t \le T)$, starting from $M(0) = 0$. This is a Hilbert space with the following inner product and norm:
$$
(M', M'') := \ME\langle M', M''\rangle_T,\ \ \mbox{and}\ \ \norm{M} := \left[\ME\langle M\rangle_T\right]^{1/2}.
$$
For each $i, k = 1, 2, \ldots$, define
$$
M_{i, k}(t) := \int_0^tg_k1\left(\mP_s(k) = i\right)\md W_i(s),\ \ t \ge 0.
$$
Then the process $M_{\infty}$ from~\eqref{eq:M-infty} can be represented as
\begin{equation}
\label{eq:M-infty-repr}
M_{\infty}(t) = \SL_{i=1}^{\infty}\SL_{k=1}^{\infty}M_{i, k}(t), \ \ t \ge 0.
\end{equation}

\begin{lemma} All processes $M_{i, k}$, $i, k = 1, 2, \ldots$, are elements of the space $\mathbb M$ and are orthogonal in this space.
\label{lemma:orthogonal}
\end{lemma}

\begin{proof} That each of these processes is a continuous square-integrable martingale is straightforward. Let us show that  $(M_{i', k'}, M_{i'', k''}) = 0$ when $i' \ne i''$ or $k' \ne k''$. Indeed, for $i' \ne i''$, this follows from the fact that the Brownian motions $W_{i'}$ and $W_{i''}$ are independent, and therefore, $\langle W_{i'}, W_{i''}\rangle_s \equiv 0$. For $i' = i'' = i$ and $k' \ne k''$, this follows from an observation that the mapping $\mP_s : \{1, 2, \ldots\} \to \{1, 2, \ldots\}$ is one-to-one for every $s \ge 0$, and therefore 
$$
1\left(\mP_s(k') = i\right)1\left(\mP_s(k'') = i\right) \equiv 0.
$$
\end{proof}

It is easy to see that 
\begin{equation}
\label{eq:sum-of-norms}
\SL_{i=1}^{\infty}\SL_{k=1}^{\infty}\norm{M_{i, k}}^2 = \SL_{i=1}^{\infty}\SL_{k=1}^{\infty}\int_0^Tg_k^21\left(\mP_s(k) = i\right)\md s = T\SL_{k=1}^{\infty}g_k^2 = TG. 
\end{equation}
From~\eqref{eq:sum-of-norms} and Lemma~\ref{lemma:orthogonal}, we get that the series~\eqref{eq:M-infty-repr} converges in the space $\mathbb M$, which proves that $M_{\infty}$ is a continuous square-integrable martingale. The calculation similar to the one in~\eqref{eq:sum-of-norms} with $t$ instead of $T$ shows that $\langle M_{\infty}\rangle_t \equiv Gt$. This completes the proof of Theorem~\ref{thm:Girsanov-existence}.

\subsection{Proof of Lemma~\ref{names2ranks}} Parts of this result were already proved in \cite{IKS2013} for (slightly more restrictive) conditions of Theorem~\ref{existenceinfclassical}. We can write each $X_i$ in the form
$$
X_i(t) = x_i + \int_0^t\be_i(s)\md s + \int_0^t\rho_i(s)\md W_i(s),\ \ t \ge 0,
$$
where the drift and diffusion coefficients
$$
\be_i(t) = \SL_{k = 1}^{\infty}1(\mP_t(k) = i)g_k,\ \ \rho_i(t) = \SL_{k=1}^{\infty}1(\mP_t(k) = i)\si_k
$$
satisfy the following inequalities:
$$
|\be_i(t)| \le \ol{g},\ \ |\rho_i(t)| \le \ol{\si},\ \ 0 \le t \le T.
$$
There exists a random but a.s. finite $i_0$ such that for $i \ge i_0$ we have: $x_i > \ol{g}T + u$. For these $i$, by Lemma~\ref{lemma:estimate-special} we have:
$$
\MP\left(\min\limits_{0 \le t \le T}X_i(t) \le u\right) \le 2\Psi\left(\frac{x_i - \ol{g}T - u}{\ol{\si}\sqrt{T}}\right).
$$
Apply Lemma~\ref{lemma:special-series} and the Borel-Cantelli lemma and finish the proof of the local finiteness. Now, let us show that a.s. there exist only finitely many $i$ such that $X_i(t) \le x_i/2$. There exists a random but a.s. finite $i_1$ such that for $i \ge i_1$ we have: $x_i/2 > \ol{g}T$. Then $x_i > x_i/2 + \ol{g}T$ for these $i$. For $i \ge i_0\vee i_1$, by Lemma~\ref{lemma:estimate-special} we have:
$$
\MP(X_i(t) \le x_i/2) \le \MP\left(\min\limits_{0 \le s \le t}X_i(s) \le x_i/2\right) \le 2\Psi\left(\frac{x_i - x_i/2 - \ol{g}T}{\ol{\si}\sqrt{T}}\right).
$$
Apply Lemma~\ref{lemma:special-series} and the Borel-Cantelli lemma. This proves that there exists a random but a.s. finite $i_2 \ge i_0\vee i_1$ such that $X_i(t) \ge x_i/2$ for $i \ge i_2$. Thus, for $i \ge i_2$, we have: $X_i(t) \ge x_i/2 \ge 0$, and almost surely, we get:
$$
\SL_{i=i_2}^{\infty}e^{-\al X_i(t)^2} \le \SL_{i=i_2}^{\infty} e^{-\al(x_i/2)^2} < \infty.
$$
Because $i_2$ is a.s. finite, this completes the proof. 

\subsection{Proof of Lemma~\ref{lemma:local-time}} This statement follows from similar statement for finite systems (see~\eqref{symmSDE}). Indeed, take the $k$th ranked particle $Y_k$ and let $u := \max_{[0, T]}Y_k + 1$. Let us show that for every $t \in [0, T]$ there exists a (possibly random) neighborhood of $t$ in $[0, T]$ such that ~\eqref{dynamics} holds. The statement of Lemma~\ref{lemma:local-time} would then follow from compactness of $[0, T]$ and the fact that $T > 0$ is arbitrary. 

Indeed, there exists an $i_0$ such that $\min_{[0, T]}X_i > u$ for $i > i_0$. Take the minimal such $i_0$. Then, take $m > k$ and assume the event $\{i_0 \le m\}$ happened. Fix time $t \in [0, T]$. We claim that if $Y_k$ does not collide at time $t$ with other particles, then there exists a (random) neighborhood when $Y_k$ does not collide with other particles. Indeed, particles $X_i,\ i > m$, cannot collide with $Y_k$, by definition of $u$ and $i_0$. And for every particle $X_i,\ i = 1, \ldots, m$, other than $Y_k$ (say $Y_k$ has name $j$ at time $t$), there exists a (random) open time neighborhood of $t$ such that this particle does not collide with $Y_k = X_j$ in this neighborhood. Take the finite intersection of these $m - 1$ neighborhoods and complete the proof of the claim. In this case, the formula~\eqref{dynamics} is trivial, because the local time terms $L_{(k-1, k)}$ and $L_{(k, k+1)}$ are constant in this neighborhood. 

Now, if $Y_k(t)$ does collide with particles $X_i,\ i \in I$, then $I \subseteq \{1, \ldots, m\}$. We claim that there exists a neighborhood of $t$ such that, in this neighborhood, the
particles $X_i,\ i \in I$, do not collide with any other particles. Indeed, for every $i \in I$, we have: $X_i(t) = Y_k(t) \le u - 1$. There exists a neighborhood of $t$ in which 
$X_i$ does not collide with any particles $X_l,\ l \in \{1, \ldots, m\}\setminus I$. There exists another neighborhood in which $X_i(t) < u$. Therefore, $X_i$ does not collide with any particles $X_l,\ l > m$. Intersect all these neighborhoods (there are $2|I|$ of them) and complete the proof of this claim. In this neighborhood, the system $(X_i)_{i \in I}$ behaves as a finite system of competing Brownian particles. It suffices to refer to~\eqref{symmSDE}.

\subsection{Proof of Theorem~\ref{thm2}} {\bf Step 1.} {\it $q^+_n \ge 1/2$ for all $n \ge 1$.} For $N \ge 2$, consider a ranked system 
$$
Y^{(N)} = \left(Y^{(N)}_1, \ldots, Y^{(N)}_N\right)',
$$
of $N$ competing Brownian particles, with parameters 
$$
(g_n)_{1 \le n \le N},\ \  (\si_n^2)_{1 \le n \le N},\ \  (q^{\pm}_n)_{1 \le n \le N},
$$ 
starting from $Y^{(N)}_k(0) = y_k,\ k = 1, \ldots, N$, with driving Brownian motions $B_1, B_2, \ldots, B_N$. Define the new parameters of collision
$$
\ol{q}^{\pm}_n = \frac12,\ n \ge 1.
$$
Consider another ranked system 
$$
\ol{Y}^{(N)} = \left(\ol{Y}^{(N)}_1, \ldots, \ol{Y}^{(N)}_N\right)',
$$
of $N$ competing Brownian particles, with parameters 
$$
(g_n)_{1 \le n \le N},\ \  (\si_n^2)_{1 \le n \le N},\ \  \left(\ol{q}^{\pm}_n\right)_{1 \le n \le N},
$$
starting from the same initial conditions $Y^{(N)}_k(0) = \ol{Y}^{(N)}_k(0) = y_k,\ k = 1, \ldots, N$, with the same driving Brownian motions $B_1, B_2, \ldots, B_N$. We can construct such a system in the strong sense, by result o f Section 2 and \cite{KPS2012} so that the sequences of driving Brownian motions $(B_1, \ldots, B_N)$ for each $N$ are nested into each other. By \cite[Corollary 3.9]{MyOwn2}, for $k = 1, \ldots, N$ and $t \ge 0$, we have:
\begin{equation}
\label{659}
\ol{Y}^{(N+1)}_k(t) \le \ol{Y}^{(N)}_k(t),\ \ Y^{(N+1)}_k(t) \le Y^{(N)}_k(t).
\end{equation}
Since $q^+_n \ge \ol{q}^+_n = 1/2$ for $n = 1, \ldots, N$, by \cite[Corollary 3.12]{MyOwn2}, we have: 
\begin{equation}
\label{109}
\ol{Y}^{(N)}_k(t) \le Y^{(N)}_k(t),\ \ t \ge 0,\ \ k = 1, \ldots, N.
\end{equation}

\begin{lemma}
For every $T > 0$, we have a.s.
$$
\lim\limits_{N \to \infty}\min\limits_{0 \le t \le T}\ol{Y}^{(N)}_1(t) = 
\inf\limits_{N \ge 2}\min\limits_{0 \le t \le T}\ol{Y}^{(N)}_1(t) > -\infty.
$$
\label{lemma:finite-limit}
\end{lemma}
The proof of Lemma~\ref{lemma:finite-limit} is postponed until the end of the proof of Theorem~\ref{thm2}. This lemma is used for the pathwise lower bound of the sequence $(\ol{Y}^{(N)}_1)_{N \ge 2}$ of processes. Assuming we proved this lemma, let us continue the proof of Theorem~\ref{thm2}. 

\medskip

{\bf Step 2.} Note that for all $s \in [0, T]$, 
$$
\min\limits_{0 \le t \le T}\ol{Y}^{(N)}_1(t) \le \ol{Y}^{(N)}_1(s).
$$
Therefore, by Lemma~\ref{lemma:finite-limit}, for every $k \ge 1$, $t \ge 0$, $N \ge k$, we have:
$$
Y^{(N)}_k(t) \ge \ol{Y}^{(N)}_k(t) \ge \ol{Y}^{(N)}_1(t) \ge \lim\limits_{N \to \infty}\min\limits_{0 \le t \le T}\ol{Y}^{(N)}_1(t).
$$
By~\eqref{659}, there exists a finite pointwise limit 
\begin{equation}
\label{eq:Y-k-limit}
Y_k(t) := \lim\limits_{N \to \infty}Y^{(N)}_k(t).  
\end{equation}
Now, let $L^{(N)} = \left(L^{(N)}_{(1, 2)}, \ldots, L^{(N)}_{(N-1, N)}\right)'$ be the vector of local times for the system $Y^{(N)}$. 

\begin{lemma}
\label{lemma:conv-of-local-times}
There exist a.s. continuous limits
$$
L_{(k, k+1)}(t) := \lim\limits_{N \to \infty}L^{(N)}_{(k, k+1)}(t),
$$
for each $k \ge 1$, uniform on every $[0, T]$. The limit $Y_k(t)$ from~\eqref{eq:Y-k-limit} is also continuous and uniform on every $[0, T]$ for every $k \ge 1$. 
\end{lemma}

The proof of Lemma~\ref{lemma:conv-of-local-times} is also postponed until the end of the proof of Theorem~\ref{thm2}. Assuming we proved this lemma, let us complete the proof of Theorem~\ref{thm2} for the case when $q^+_n \ge 1/2$ for all $n \ge 1$. For $k = 1, 2, \ldots$ and $t \ge 0$, we have:
$$
Y_k^{(N)}(t) = y_k + g_kt + \si_kB_k(t) + q^+_kL_{(k-1, k)}^{(N)}(t) - q^-_kL_{(k, k+1)}^{(N)}(t).
$$
Letting $N \to \infty$, we have:
$$
Y_k(t) = y_k + g_kt + \si_kB_k(t) + q^+_kL_{(k-1, k)}(t) - q^-_kL_{(k, k+1)}(t).
$$
Finally,  let us show that $L_{(k, k+1)}$ and $Y_k$ satisfy the properties (i) - (iii) of Definition~\ref{main}. Some of these properties follow directly from the uniform covergence and the corresponding properties for finite systems $Y^{(N)}$. The nontrivial part is to prove that $L_{(k, k+1)}$ can increase only when $Y_k = Y_{k+1}$. Suppose that for some $k \ge 1$ we have: $Y_k(t) < Y_{k+1}(t)$ for $t \in [\al, \be] \subseteq \BR_+$. By continuity, there exists $\eps > 0$ such that $Y_{k+1}(t) - Y_k(t) \ge \eps$ for $t \in [\al, \be]$. By uniform convergence, there exists an (a.s. finite) $N_0$ such that for $N \ge N_0$ we have:
$$
Y^{(N)}_{k+1}(t) - Y^{(N)}_k(t) \ge \frac{\eps}2,\ \ t \in [\al, \be].
$$
Therefore, $L^{(N)}_{(k, k+1)}$ is constant on $[\al, \be]$: $L^{(N)}_{(k, k+1)}(\al) = L^{(N)}_{(k, k+1)}(\be)$. This is true for all $N \ge N_0$. Letting $N \to \infty$, we get: $L_{(k, k+1)}(\al) = L_{(k, k+1)}(\be)$. Therefore, $L_{(k, k+1)}$ is also constant on $[\al, \be]$.

\medskip

{\bf Step 3.} Now, consider the case when $q^+_n \ge 1/2$ only for $n \ge n_0$. It suffices to show that the sequence $(Y^{(N)}_k(t))_{N \ge k}$ is bounded from below, since this is the crucial part of the proof. For $N \ge n_0+2$, consider the system 
$$
\tilde{Y}^{(N)} = \left(\tilde{Y}^{(N)}_{n_0+1}, \ldots, \tilde{Y}^{(N)}_N\right)'
$$
of $N - n_0$ competing Brownian particles with parameters 
$$
(g_n)_{n_0 < n \le N},\ (\si_n^2)_{n_0 < n \le N},\ (q^{\pm}_n)_{n_0 < n \le N},
$$
starting from $(y_{n_0+1}, \ldots, y_N)'$, with driving Brownian motions $B_{n_0+1}, \ldots, B_N$. By \cite[Corollary 3.9, Remark 8]{MyOwn2}, we have:
\begin{equation}
\label{eq:comp-of-two-systems}
Y^{(N)}_k(t) \ge \tilde{Y}^{(N)}_k(t),\ \ \mbox{for}\ \ n_0 < k \le N\ \ \mbox{and}\ \ t \ge 0.
\end{equation}
But for every $k > n_0$ and $t \in [0, T]$, the sequence $(\tilde{Y}^{(N)}_k(t))_{N > k}$ is bounded below: we proved this earlier in the proof of Theorem~\ref{thm2}, thanks to Lemma~\ref{lemma:finite-limit} and~\eqref{109}.  
Let us show that for every $t \in [0, T]$, the sequence $(Y^{(N)}_1(t))_{N \ge 2}$ is bounded below. Indeed, again applying \cite[Corollary 3.9]{MyOwn2}, we get:
$$
Z^{(n_0+1)}_k(t) \ge Z^{(N)}_k(t),\ \ t \ge 0,\ k = 1, \ldots, n_0,\ N \ge n_0+2.
$$
Note that $(Y^{(N)}_{n_0+1}(t))_{N \ge n_0+2}$ is bounded from below, and $Z^{(n_0+1)}_k(t)$ for $k = 1,\ldots, n_0$ are independent of $N$.
Combining this with 
$$
Y^{(N)}_1(t) = Y_{n_0+1}^{(N)}(t) - Z^{(N)}_{n_0}(t) - \ldots - Z_1^{(N)}(t) \ge Y^{(N)}_{n_0+1}(t) - Z_1^{(n_0+1)}(t) - \ldots - Z_{n_0}^{(n_0+1)}(t),
$$
we get that $(Y^{(N)}_1(t))_{N \ge 2}$ is bounded from below. The rest of the proof is the same as in the case when $q^+_n \ge 1/2$ for all $n = 1, 2, \ldots$

\medskip

{\it Proof of Lemma~\ref{lemma:finite-limit}.} It suffices to show that, as $u \to \infty$, we have:
$$
\sup\limits_{N \ge 2}\MP\left(\min\limits_{0 \le t \le T}\ol{Y}_1^{(N)}(t) < -u\right) \to 0.
$$
The ranked system $\ol{Y}^{(N)}$ has the same law as the result of ranking of a classical system 
$$
X^{(N)} = \left(X^{(N)}_1, \ldots, X^{(N)}_N\right)'
$$
with the same parameters: drift coefficients $(g_n)_{1 \le n \le N}$, diffusion coefficients $(\si_n^2)_{1 \le n \le N}$, starting from $X^{(N)}(0) = (y_1, \ldots, y_N)'$. These components satisfy the following system of SDE:
\begin{equation}
\label{eq:SDE-named}
\md X_i^{(N)}(t) = \SL_{k=1}^N1(X_i^{(N)}\  \mbox{has rank}\  k\  \mbox{at time}\  t)\left(g_k\md t + \si_k\md W_i(t)\right),
\end{equation}
for some i.i.d. standard Brownian motions $W_1, \ldots, W_N$. In particular,
$$
Y_1^{(N)}(t) \equiv \min\limits_{i = 1, \ldots, N}X_i^{(N)}(t).
$$
Therefore,
\begin{equation}
\label{eq:min-min}
\min\limits_{0 \le t \le T}Y_1^{(N)}(t) = \min\limits_{1 \le i \le N}\min\limits_{0 \le t \le T}X_i^{(N)}(t).
\end{equation}
We can rewrite~\eqref{eq:SDE-named} as
$$
X_i^{(N)}(t) = y_i + \int_0^t\be_{N, i}(s)\md s + \int_0^t\rho_{N, i}(s)\md W_i(s),
$$
where
$$
\be_{N, i}(t) := \SL_{k=1}^Ng_k1(X_i^{(N)}\ \mbox{has rank}\ k\ \mbox{at time}\  t),
$$
$$
\rho_{N, i}(t) := \SL_{k=1}^N\si_k1(X_i^{(N)}\ \mbox{has rank}\ k\ \mbox{at time}\  t).
$$
Because of~\eqref{eq:bounds-g-si}, we have the following estimates: $\be_{N, i}(t) \ge \underline{g}$ and $|\rho_{N, i}(t)| \le \ol{\si}$, for $t \ge 0$. Therefore, by Lemma~\ref{lemma:estimate-special} we get: 
$$
\MP\left(\min\limits_{0 \le t \le T}X_i^{(N)}(t) < -u\right) \le 
2\Psi\left(\frac{u+y_i-(\underline{g}T)_-}{\ol{\si}\sqrt{T}}\right).
$$
From~\eqref{eq:min-min}, we have:
\begin{equation}
\label{eq:1342}
\MP\left(\min\limits_{0 \le t \le T}\ol{Y}^{(N)}_1(t) < -u\right) \le 2\SL_{i=1}^N\Psi\left(\frac{u+y_i-(\underline{g}T)_-}{\ol{\si}\sqrt{T}}\right). 
\end{equation}
By Lemma~\ref{lemma:special-series}, we have:
\begin{equation}
\label{eq:1341}
\SL_{N=1}^{\infty}\SL_{i=1}^N\Psi\left(\frac{u+y_i-(\underline{g}T)_-}{\ol{\si}\sqrt{T}}\right) < \infty.
\end{equation}
Comparing~\eqref{eq:1342} and~\eqref{eq:1341}, we get: 
$$
\sup\limits_{N\ge 2}\MP\left(\min\limits_{0 \le t \le T}\ol{Y}_1^{(N)}(t) < -u\right) < \infty.
$$
Let $u \to \infty$. Then 
$$
\frac{y_i + (\underline{g}T)_- + u}{\ol{\si}\sqrt{T}} \to \infty,\ \ \Psi\left(\frac{y_i + (\underline{g}T)_- + u}{\ol{\si}\sqrt{T}}\right) \to 0.
$$
Applying Lebesgue dominated convergence theorem to this series (and using the fact that $\Psi$ is decreasing), we get:
$$
\SL_{i=1}^{\infty}\Psi\left(\frac{u+y_i+(\underline{g}T)_-}{\ol{\si}\sqrt{T}}\right) \to 0\ \ \mbox{as}\ \ u \to \infty.
$$
This completes the proof of Lemma~\ref{lemma:finite-limit}.

\medskip

{\it Proof of Lemma~\ref{lemma:conv-of-local-times}.} Applying \cite[Corollary 3.9]{MyOwn2}, we have: for $0 \le s \le t$ and $1 \le k < N < M$, 
\begin{equation}
\label{812}
L^{(N)}_{(k, k+1)}(t) - L^{(N)}_{(k, k+1)}(s) \le L^{(M)}_{(k, k+1)}(t) - L^{(M)}_{(k, k+1)}(s).
\end{equation}
By construction of these systems, the initial conditions $y_k = Y^{(N)}_k(0),\ N \ge k$, do not depend on $N$. Therefore,
$$
Y^{(N)}_1(t) = y_1 + g_1t + \si_1B_1(t) - q^-_1L_{(1, 2)}^{(N)}(t).
$$
Since $Y^{(N)}_1(t) \to Y_1(t)$ and $q^-_1 > 0$: the sequence $(L^{(N)}_{(1, 2)}(t))_{N \ge 2}$ has a limit 
$$
L_{(1, 2)}(t) := \lim\limits_{N \to \infty}L^{(N)}_{(1, 2)}(t),\ \ \mbox{for every}\ \ t \ge 0.
$$
Letting $M \to \infty$ in~\eqref{812}, we get: for $t \ge s \ge 0$, 
$$
L_{(1, 2)}(t) - L_{(1, 2)}(s) \ge L_{(1, 2)}^{(N)}(t) - L_{(1, 2)}^{(N)}(s).
$$
We can equivalently rewrite this as
\begin{equation}
\label{31871}
L_{(1, 2)}(t) - L_{(1, 2)}^{(N)}(t) \ge L_{(1, 2)}(s) - L_{(1, 2)}^{(N)}(s).
\end{equation}
But we also have: $(L^{(N)}_{(1, 2)}(t))_{N \ge 2}$ is nondecreasing. Therefore, 
\begin{equation}
\label{31872}
L_{(1, 2)}(s) - L_{(1, 2)}^{(N)}(s) \ge 0.
\end{equation}
In addition, we get the following convergence:
\begin{equation}
\label{31873}
L^{(N)}_{(1, 2)}(t) \to L_{(1, 2)}(t)\ \ \mbox{as}\ \ N \to \infty.
\end{equation}
Combining~\eqref{31871}, ~\eqref{31872}, ~\eqref{31873}, we get:
$$
\lim\limits_{N \to \infty}L^{(N)}_{(1, 2)}(s) = L_{(1, 2)}(s)\ \ \mbox{uniformly on every}\ \ [0, t].
$$
Therefore, letting $N \to \infty$ in~\eqref{659}, we get:
$$
Y_1(t) = y_1 + g_1t + \si_1B_1(t) - q^-_1L_{(1, 2)}(t),\ \ t \ge 0,
$$
and $Y_1^{(N)}(s) \to Y_1(s)$ uniformly on every $[0, t]$. Since $Y_1^{(N)}$ and $L_{(1, 2)}^{(N)}$ are continuous for every $N \ge 2$, and the uniform limit of continuous functions is continuous, we conclude that the functions $Y_1$ and $L_{(1, 2)}$ are also continuous. Now, 
$$
Y_2^{(N)}(t) = y_2 + g_2t + \si_2B_2(t) + q^+_2L_{(1, 2)}^{(N)}(t) - q^-_2L_{(2, 3)}^{(N)}(t),\ \ t \ge 0.
$$
But 
$$
Y^{(N)}_2(t) \to Y_2(t)\ \ \mbox{and}\ \ L_{(1, 2)}^{(N)}(t) \to L_{(1, 2)}(t)\ \ \mbox{as}\ \ N \to \infty.
$$
Since $q^-_2 > 0$, we have: there exists a limit $L_{(2, 3)}(t) := \lim_{N \to \infty}L^{(N)}_{(2, 3)}(t)$. Similarly, we prove that this convergence is uniform on every $[0, T]$. Therefore, $\lim_{N \to \infty}Y_2^{(N)} = Y_2$ uniformly on every $[0, T]$. Thus $Y_2$ and $L_{(2, 3)}$ are continuous. Analogously, we can prove that for every $k \ge 1$, the limits 
$$
L_{(k, k+1)}(t) = \lim\limits_{N \to \infty}L_{(k, k+1)}^{(N)}(t)\ \ \mbox{and}\ \ Y_k(t) = \lim\limits_{N \to \infty}Y_k^{(N)}(t)
$$
exist and are uniform on every $[0, T]$. This completes the proof of Lemma~\ref{lemma:conv-of-local-times}, and with it the proof of Theorem~\ref{thm2}. 

\subsection{Proof of Lemma~\ref{aux}} {\bf Step 1.} First, consider the case $q^+_n \ge 1/2$ for {\it all} $n \ge 1$. Take an approximative version $\tilde{Y} = (\tilde{Y}_1, \tilde{Y}_2, \ldots)$ of the infinite classical system with parameters $(g_k)_{k \ge 1}$ and $(\si_k^2)_{k \ge 1}$, with symmetric collisions, and with the same initial conditions. By comparison techniques, Corollary~\ref{comparinfcordrifts} (iii), we have the stochastic domination:
\begin{equation}
\label{eq:101}
Y_k(t) \succeq \tilde{Y}_k(t),\ k = 1, 2, \ldots,\ 0 \le t \le T. 
\end{equation}

{\bf Step 2.} Now, let us prove the two statements for the general case.
 Consider the approximative version $\tilde{Y} = (\tilde{Y}_k)_{k > n_0}$ of the infinite ranked system of competing Brownian particles with parameters $(g_n)_{n > n_0}, (\si_n^2)_{n > n_0}, (q^{\pm}_n)_{n > n_0}$. But $q^+_n \ge 1/2$ for all $n > n_0$, and therefore the system 
$\tilde{Y}$ satisfies the statements of Lemma~\ref{aux}. By comparison techniques for infinite systems, see Corollary~\ref{cor:comp2}, we get:
$$
Y_k(t) \ge \tilde{Y}_k(t),\ \ t \in [0, T],\ \ n_0 < k \le N.
$$
Therefore, the system $(Y_k)_{k \ge 1}$ also satisfies the statements of Lemma~\ref{aux}.

\subsection{Proof of Lemma~\ref{notieinf}}

 Let $D = \{Y(t)\ \mbox{has a tie}\}$. Assume $\omega \in D$, that is, the vector $Y$ has a tie: 
\begin{equation}
\label{longtie}
Y_{k-1}(t) < Y_k(t) = Y_{k+1}(t) = \ldots = Y_l(t) < Y_{l+1}(t).
\end{equation}
This tie cannot contain infinitely many particles, because this would contradict Lemma~\ref{aux}. 
Fix a rational $q \in (Y_l(t), Y_{l+1}(t))$. By continuity of $Y_l$ and $Y_{l+1}$, there exists $M \ge 1$ such that for $s \in [t - 1/M, t + 1/M]$ we have: $Y_l(s) < q < Y_{l+1}(s)$. Let 
\begin{align*}
C(k, l, q, M) =& \biggl\{Y_{k-1}(t) < Y_k(t) = Y_{k+1}(t) = \ldots = Y_l(t) < Y_{l+1}(t),\ \ \\ &\mbox{and}\ \ Y_l(s) < q < Y_{l+1}(s)\ \ \mbox{for all}\ \ s \in \left[t - \frac1M, t + \frac1M\right]\biggr\}.
\end{align*}
We just proved that 
\begin{equation}
\label{bigunion}
\MP\left(D\setminus\bigcup\limits_{M = 1}^{\infty}\bigcup\limits_{q \in \mathbb Q}\bigcup\limits_{k < l}C(k, l, q, M)\right) = 0.
\end{equation}
Now let us show that for every $k, l , M = 1, 2, \ldots$ with $k < l$ and for every $q \in \mathbb Q$, we have: 
\begin{equation}
\label{cap}
P(D\cap C(k, l, q, M)) = 0.
\end{equation}
Since the union in~\eqref{bigunion} is countable, this completes the proof. 
If the event $C(k, l, q, M)$ happened, then we have: $([Y(u + t - 1/M)]_l, 0 \le u \le 1/M)$ behaves as a system of $l$ ranked competing Brownian particles with parameters 
$$
(g_n)_{1 \le n \le l},\ \  (\si_n^2)_{1 \le n \le l},\ \  (q^{\pm}_n)_{1 \le n \le l}.
$$
By Lemma~\ref{notie}, the probability of a tie at $t = 1/M$ for the system $([Y(u + t - 1/M)]_l, 0 \le u \le 1/M)$ of $l$ competing Brownianb particles is zero, which proves~\eqref{cap}.  

\subsection{Proof of Theorem~\ref{convnames}} 
Let $\mP_t^{(N)}$ be the ranking permutation for the vector $X^{(N)}(t) \in \BR^N$. Then for $1 \le i \le N$ we have: 
\begin{equation}
\label{34332}
X^{(N)}_i(t) = x_i + \int_0^t\be_{N, i}(s)\md s + \int_0^t\rho_{N, i}(s)\md W_{N, i}(s),\ \ t \ge 0,
\end{equation}
where $W_{N, 1}, \ldots, W_{N, N}$ are i.i.d. standard Brownian motions, 
$$
\be_{N, i}(t) = \SL_{k=1}^N1(\mP_t^{(N)}(k) = i)g_k,\ \ \mbox{and}\ \ 
\rho_{N, i}(s) = \SL_{k=1}^N1(\mP_t^{(N)}(k) = i)\si_k.
$$
Note that 
$$
\bigl|\be_{N, i}(t)\bigr| \le \max\limits_{k \ge 1}|g_k| =: \ol{g},
$$
and 
$$
\bigl|\rho_{N, i}(t)\bigr| \le \max\limits_{k \ge 1}\si_k =: \ol{\si}.
$$
Fix $T > 0$. It follows from the Arzela-Ascoli criterion and Lemma~\ref{lemma:tight-martingale} that the sequence $(X^{(N)}_i)_{N \ge i}$ is tight in $C[0, T]$.  Now, let us show that the following sequence is also tight in $C\left([0, T], \BR^{3k}\right)$, for each $k \ge 1$: 
\begin{equation}
\label{eq:subsequence}
(X_i^{(N)}, Y_i^{(N)}, W_{N, i}, i = 1, \ldots, k)_{N \ge k}.
\end{equation}
For the components $Y_i^{(N)}$, this follows from Theorem~\ref{thm3}: as $N \to \infty$, $Y_i^{(N)} \Ra Y_i$, where $Y = (Y_i)_{i \ge 1}$ is an approximative version of the infinite system of competing Brownian particles with parameters $(g_n)_{n \ge 1}$, $(\si_n^2)_{n \ge 1}$, $(\ol{q}^{\pm}_n = 1/2)_{n \ge 1}$. For the components $W_{N, i}$, this is immediate, because all these elements have the same law in $C([0, T], \BR^d)$ (the law of the $d$-dimensional Brownian motion starting from the origin). By the diagonal argument, for every subsequence $(N_m)_{m \ge 1}$ there exists a sub-subsequence $(N'_m)_{m \ge 1}$ such that for every $k \ge 1$, the following subsequence of~\eqref{eq:subsequence}
$$
(X_1^{(N'_m)}, \ldots, X_k^{(N'_m)}, Y_1^{(N'_m)}, \ldots, Y_k^{(N'_m)}, W_{N'_m, 1}, \ldots, W_{N'_m, k})_{m \ge 1}
$$
converges weakly in $C\left([0, T], \BR^{3k}\right)$. By Skorohod theorem, we can assume that the convergence is, in fact, a.s. Let 
$$
X_i := \lim\limits_{m \to \infty}X_i^{(N'_m)},\ \ Y_i := \lim\limits_{m \to \infty}Y_i^{(N'_m)},\ \ W_i := \lim\limits_{m \to \infty}W_{N'_m, i},\ \ i \ge 1
$$
be the a.s. uniform limit on $[0, T]$. As mentioned earlier, $Y = (Y_i)_{i \ge 1}$ is an approximative version of the infinite system of competing Brownian particles with parameters $(g_n)_{n \ge 1}$, $(\si_n^2)_{n \ge 1}$, $(\ol{q}^{\pm}_n = 1/2)_{n \ge 1}$. Also, $W_i$ are i.i.d. standard Brownian motions. 

\medskip

Next, it suffices to show that $X$ is a version of the infinite classical system, because the subsequence $(N_m)_{m \ge 1}$ is arbitrary, and the tightness is established above. Take the (random) set $\CN(\oa)$ of times $t \in [0, T]$ when the system $Y$ or a system $Y^{(N'_m)}$ for some $m \ge 1$ has a tie. By Lemmata~\ref{notieinf} and~\ref{notie}, there exists a set $\Oa_* \subseteq \Oa$ of measure $\MP(\Oa_*) = 1$ such that for all $\oa \in \Oa_*$, the set $\CN(\oa)$ has Lebesgue measure zero. Therefore, for every $\eps > 0$ and every $\oa \in \Oa_*$, there exists an open subset $\CU_{\eps}(\oa) \subseteq [0, T]$ with measure $\mes(\CU_{\eps}(\oa)) < \eps$ such that $\CN(\oa) \subseteq \CU_{\eps}(\oa)$. 

\begin{lemma}
\label{lemma:incl} Fix $i \ge 1$. Then for every $\oa \in \Oa_*$, there exists an $m_0(\oa)$ such that for $m \ge m_0(\oa)$ and $k \ge 1$, 
$$
\left\{t \in [0, T]\setminus\CU_{\eps}(\oa)\mid X_i(t) = Y_k(t)\right\} \subseteq \left\{t \in [0, T]\setminus\CU_{\eps}(\oa)\mid X^{(N'_m)}_i(t) = Y^{(N'_m)}_k(t)\right\}.
$$
\end{lemma}

\begin{proof} Assume the converse. Then there exists a sequence $(t_j)_{j \ge 1}$ in $[0, T] \subseteq \CU_{\eps}(\oa)$ and a sequence $(m_j)_{j \ge 1}$ such that $m_j \to \infty$ and 
$$
X_i(t_j) = Y_k(t_j),\ \ X^{(N'_{m_j})}_i(t_j) \ne Y^{(N'_{m_j})}_k(t_j).
$$
Therefore, the particle with name $i$ in the system $X^{(N'_{m_j})}$ has rank other than $k$: either larger than $k$, in which case we have:
\begin{equation}
\label{eq:ineq65}
X^{(N'_{m_j})}_i(t_j) \ge Y^{(N'_{m_j})}_{k+1}(t_j),
\end{equation}
or smaller than $k$, in which case
\begin{equation}
\label{eq:ineq56}
X^{(N'_{m_j})}_i(t_j) \le Y^{(N'_{m_j})}_{k-1}(t_j).
\end{equation}
By the pigeonhole principle, at least one of these inequalities is true for infinitely many $j$. Without loss of generality, we can assume that~\eqref{eq:ineq65} holds for infinitely many $j \ge 1$; the case when~\eqref{eq:ineq56} holds for infinitely many $j \ge 1$ is similar. Again, without loss of generality we can assume~\eqref{eq:ineq65} holds for all $j \ge 1$. There exists a convergent subsequence of $(t_j)_{j \ge 1}$, because $[0, T]$ is compact. Without loss of generality, we can assume $t_j \to t_0$. We shall use the principle: if $f_n \to f_0$ uniformly on $[0, T]$ and $s_n \to s_0$, then $f_n(s_n) \to f_0(s_0)$. Since
$$
X^{(N'_{m_j})}_i(t_j) \to X_i(t_0)\ \ \mbox{and}\ \ Y^{(N'_{m_j})}_{k+1}(t_j) \to Y_{k+1}(t_0)
$$
uniformly on $[0, T]$, we have after letting $j \to \infty$: $X_i(t_0) \ge Y_{k+1}(t_0)$. 
But we can also let $j \to \infty$ in $X_i(t_j) = Y_k(t_j)$. We get: $X_i(t_0) = Y_k(t_0)$. 
Thus, $Y_{k+1}(t_0) \le Y_k(t_0)$. The reverse inequality always holds true. Therefore, there is a tie at the point $t_0$. But the set $[0, T] \setminus \CU_{\eps}$ is closed; therefore, $t_0 \in [0, T] \setminus \CU_{\eps}$. This contradiction completes the proof. 
\end{proof}

\begin{lemma} For $\oa \in \Oa_*$, $t \in [0, T]\setminus\CN(\oa)$, and $i \ge 1$, as $m \to \infty$, we have:
$$
\be_{N'_m, i}(t) \to \be_i(t) := \SL_{k=1}^{\infty}1(Y_k(t) = X_i(t))g_k,\ \ \mbox{and}\ \ \rho_{N'_m, i}(t) \to \rho_i(t) := \SL_{k=1}^{\infty}1(Y_k(t) = X_i(t))\si_k.
$$
\label{lemma:aux-aux-0}
\end{lemma}

\begin{proof} Let us prove the first convergence statement; the second statement is proved similarly. 
By Lemma~\ref{lemma:incl}, we have: 
$$
\be_{N'_m, i}(t) = \be_i(t)\ \ \mbox{and}\ \ \rho_{N'_m, i}(t) = \rho_i(t),\ \ t \in [0, T]\setminus\CU_{\eps},\ \ m > m_0.
$$
This proves that 
$$
\be_{N'_m, i}(t) \to \be_i(t)\ \ \mbox{and}\ \ \rho_{N'_m, i}(t) \to \rho_i(t)\ \ \mbox{for}\ \ t \in [0, T]\setminus\CU_{\eps}\ \ \mbox{as}\ \ m \to \infty.
$$
Since the set $\mes(\CU_{\eps}) < \eps$ and $\eps$ is arbitrarily small, this proves Lemma~\ref{lemma:aux-aux-0}. 
\end{proof}

Now, let us return to the proof of Theorem~\ref{convnames}. Fix $t \in [0, T]$. Apply \cite[Lemma 7.1]{MyOwn8} to show that in $L^2(\Oa, \CF, \MP)$, we have:
\begin{equation}
\label{eq:!}
\int_0^t\rho_{N'_m, i}(s)\md W_{N'_m, i}(s) \to \int_0^t\rho_i(s)\md W_i(s).
\end{equation}
Also, by Lebesgue dominated convergence theorem (because $\mes(\CN(\oa)) = 0$ for $\oa \in \Oa_*$), 
\begin{equation}
\label{eq:conv20}
\int_0^t\be_{N'_m, i}(s)\md s \to \int_0^t\be_i(s)\md s\ \ \mbox{a.s. for all}\ \ t \in [0, T].
\end{equation}
Finally, we have a.s. 
\begin{equation}
\label{eq:conv-of-ito0}
X_i^{(N'_m)}(t) = x_i + \int_0^t\be_{N'_m, i}(s)\md s + \int_0^t\rho_{N'_m, i}(s)\md W_{N'_m, i}(s) \to X_i(t). 
\end{equation}
From~\eqref{eq:conv-of-ito0} and~\eqref{eq:conv20} we have that 
\begin{equation}
\label{eq:!!}
\int_0^t\rho_{N'_m, i}(s)\md W_{N'_m, i}(s) \to X_i(t) - x_i - \int_0^t\be_i(s)\md s.
\end{equation}
But if a sequence of random variables converges to one limit in $L^2$ and to another limit a.s., then there limits coincide a.s. Comparing~\eqref{eq:!} and~\eqref{eq:!!}, we get: 
$$
X_i(t) = x_i + \int_0^t\be_i(s)\md s + \int_0^t\rho_i(s)\md W_i(s),
$$
which is another way to write the SDE governing the infinite classical system. We have found a sequence $(N'_m)_{m \ge 1}$ which corresponds to convergence on $[0, T]$. By taking a sequence $T_j \to \infty$ and using the standard diagonal argument, we can finish the proof.

\subsection{Proof of Lemma~\ref{lemma:independent-sequence}.} Because of symmetry of $\pi$ and $\tilde{\pi}$, it suffices to show that $\pi \preceq \tilde{\pi}$. Next, it suffices to show that for every fixed $M \ge 1$ we have: 
\begin{equation}
\label{eq:777comp}
[\pi]_M \preceq [\tilde{\pi}]_M.
\end{equation}
Recall that we have the following weak convergence:
$$
[\pi^{(\tilde{N}_j)}]_M \Ra [\tilde{\pi}]_M,\ \ j \to \infty,
$$
and the stochastic comparison is preserved under weak limits. Therefore, to show~\eqref{eq:777comp}, it suffices to prove that 
\begin{equation}
\label{eq:1}
[\pi]_M \preceq [\pi^{(\tilde{N}_j)}]_M.
\end{equation}
Now, take $J$ large enough so that $N_J > \tilde{N}_j$. By \cite[Corollary 3.14]{MyOwn2}, we have:
\begin{equation}
\label{eq:2}
[\pi^{(\tilde{N}_j)}]_M \succeq [\pi^{(N_J)}]_M.
\end{equation}
By construction of $\pi$, we get:
\begin{equation}
\label{eq:3}
[\pi]_M \preceq [\pi^{(N_J)}]_M.
\end{equation}
From~\eqref{eq:2} and~\eqref{eq:3}, we get~\eqref{eq:1}. 


\subsection{Proof of Theorem~\ref{thm3}} Using the notation of Theorem~\ref{thm2}, we have: 
$$
Y^{(N_j)}_k \to Y_k,\ \ j \to \infty,
$$
for every $k \ge 1$, uniformly on every $[0, T]$. Now, let 
$$
\ol{Y}^{(N_j)} = \left(\ol{Y}^{(N_j)}_1, \ldots, \ol{Y}^{(N_j)}_{N_j}\right)'
$$
be the ranked system of $N_j$ competing Brownian particles, which has the same parameters and driving Brownian motions as 
$$
Y^{(N_j)} = \left(Y^{(N_j)}_1, \ldots, Y^{(N_j)}_{N_j}\right)',
$$
but starts from 
$$
(0, z_1^{(N_j)}, z_1^{(N_j)}+z_2^{(N_j)}, \ldots, z_1^{(N_j)}+z_2^{(N_j)}+ \ldots + z_{N_j-1}^{(N_j)})',
$$
rather than $(0, z_1, z_1 + z_2, \ldots, z_1 + z_2 + \ldots + z_{N_j-1})'$. In other words, the gap process $\ol{Z}^{(N_j)}$ of the system $\ol{Y}^{(N_j)}$ is in its stationary regime: $\ol{Z}^{(N_j)}(t) \sim \pi^{(N_j)}$, $t \ge 0$. Now, let us state an auxillary lemma; its proof is postponed until the end of the proof of Theorem~\ref{thm3}. 

\begin{lemma}
\label{lemma:conv4} Almost surely, as $j \to \infty$, for all $t \ge 0$ and $k \ge 1$, we have:
\begin{equation}
\label{conv4}
Y_k(t) = \lim\limits_{j \to \infty}\ol{Y}^{(N_j)}_k(t).
\end{equation}
\end{lemma}

Assuming that we have already shown Lemma~\ref{lemma:conv4}, we can finish the proof. For every $t \ge 0$ and $k = 1, 2, \ldots$, a.s. 
$$
\ol{Z}^{(N_j)}_k(t) = \ol{Y}^{(N_j)}_{k+1}(t) - \ol{Y}^{(N_j)}_k(t)\to Z_k(t) = Y_{k+1}(t) - Y_k(t),\ \ j \to \infty. 
$$
Therefore, for every $t \ge 0$ and $M \ge 1$, a.s. we have:
$$
\left(\ol{Z}^{(N_j)}_1(t), \ldots, \ol{Z}^{(N_j)}_{M}(t)\right)' \to \left(Z_1(t), \ldots, Z_{M}(t)\right)',\ \ j \to \infty.
$$
But 
$$
\ol{Z}^{(N_j)}(t) = \left(\ol{Z}^{(N_j)}_1(t), \ldots, \ol{Z}^{(N_j)}_{N_j-1}(t)\right)' \sim \pi^{(N_j)}
$$
for $j \ge 1$ and $t \ge 0$. Moreover, as $j \to \infty$, we have the following weak convergence:
$$
[\pi^{(N_j)}]_{M}\ \Ra\ [\pi]_{M}.
$$
Therefore, for $M \ge 1$, $t \ge 0$, we get:
$$
\left(Z_1(t), \ldots, Z_{M}(t)\right)' \sim [\pi]_{M}.
$$
Thus, for $Z(t) := (Z_1(t), Z_2(t), \ldots)$, we have: 
$$
Z(t) \sim \pi,\ \ t \ge 0.
$$

\medskip

{\it Proof of Lemma~\ref{lemma:conv4}.} First, since $z_1 \le z_1^{(N_j)}, \ldots, z_{N_j-1} \le z_{N_j-1}^{(N_j)}$, 
we have: 
\begin{align*}
Y^{(N_j)}(0) =& (0, z_1, z_1+z_2, \ldots, z_1+z_2+\ldots + z_{N_j-1})'  \\ & \le \ol{Y}^{(N_j)}(0) = (0, z_1^{(N_j)}, z_1^{(N_j)}+z_2^{(N_j)}, \ldots, z_1^{(N_j)}+z_2^{(N_j)} + \ldots + z_{N_j-1}^{(N_j)})'.
\end{align*}
By \cite[Corollary 3.11(i)]{MyOwn2}, 
\begin{equation}
\label{65676}
Y^{(N_j)}_k(t) \le \ol{Y}^{(N_j)}_k(t),\ t \ge 0,\ j \ge 1.
\end{equation}
As shown in the proof of Theorem~\ref{thm2}, 
\begin{equation}
\label{65677}
Y^{(N_j)}_k(t) \ge Y_k(t),\ \ k = 1, \ldots, N_j,\ t \ge 0.
\end{equation}
Combining~\eqref{65676} and~\eqref{65677}, we get:
\begin{equation}
\label{751}
Y_k(t) \le \ol{Y}^{(N_j)}_k(t),\ \ k = 1, \ldots, N_j,\ \ t \ge 0.
\end{equation}
On the other hand, fix $\eps > 0$ and $j \ge 1$. Then 
$\lim\limits_{l \to \infty}z^{(N_l)}_k = z_k$, for $k = 1, \ldots, N_j-1$. 
There exists an $l_0(j, \eps)$ such that for $l > l_0(j, \eps)$ and $k = 1, \ldots, N_j-1$, 
$$
\label{750}
z_1^{(N_l)}+\ldots + z_{k}^{(N_l)} \le z_1 + \ldots + z_{k} + \eps.
$$
For such $l$, let $\check{Y} = (\check{Y}_1, \ldots, \check{Y}_{N_j})'$, be another system of $N_j$ competing Brownian particles, with the same parameters and driving Brownian motions, as $Y^{(N_j)}$, but starting from 
$(0, z_1^{(N_l)}, z_1^{(N_l)}+z_2^{(N_l)}, \ldots, z_1^{(N_l)}+z_2^{(N_l)}+\ldots + z_{N_j-1}^{(N_l)})'$. 
By \cite[Corollary 3.9]{MyOwn2}, 
\begin{equation}
\label{7961}
\check{Y}_k(t) \ge \ol{Y}^{(N_l)}_k(t),\ \ k = 1, \ldots, N_j,\ \ t \ge 0,
\end{equation}
since $\check{Y}$ is obtained from $\ol{Y}^{(N_l)}$ by removing the top $N_l - N_j$ particles. However, 
$$
Y^{(N_j)} + \eps\mathbf{1}_{N_j} := (Y^{(N_j)}_1 + \eps, \ldots, Y^{(N_j)}_{N_j} + \eps)',
$$
is also a system of $N_j$ competing Brownian particles, with the same parameters and driving Brownian motions as $Y^{(N_j)}$, but starting from 
$(\eps, z_1 + \eps, \ldots, z_1+\ldots + z_{N_j-1} + \eps)'$. Since $Y^{(N_j)}(0) + \eps \ge \check{Y}(0)$, because of~\eqref{750}, by \cite[Corollary 3.11(i)]{MyOwn2}, we have:
\begin{equation}
\label{7962}
\check{Y}_k(t) \le Y^{(N_j)}_k(t) + \eps,\ \ k = 1, \ldots, N_j,\ \ t \ge 0.
\end{equation}
Combining~\eqref{7961} and~\eqref{7962}, we get: 
$\ol{Y}^{(N_l)}_k(t) \le Y^{(N_j)}_k(t) + \eps$, for $k = 1, \ldots, N_j$, and $t \ge 0$. But for every fixed $k = 1, 2, \ldots$, $\lim_{j \to \infty}Y^{(N_j)}_k(t) = Y_k(t)$. Therefore, there exists $j_0(k) \ge 2$ such that $Y^{(N_{j_0(k)})}_k(t) \le Y_k(t) + \eps$. Meanwhile, for $l > l_0(j_0(k), k)$ we get: 
\begin{equation}
\label{7963}
\ol{Y}^{(N_l)}_k(t) \le Y_k(t) + 2\eps.
\end{equation}
We also have from~\eqref{751} that 
\begin{equation}
\label{7964}
\ol{Y}^{(N_l)}_k(t) \ge Y_k(t).
\end{equation}
Since $\eps > 0$ is arbitrarily small, combining~\eqref{7963} and~\eqref{7964}, we get~\eqref{conv4}.

\subsection{Proof of Lemma~\ref{satisfaction}.} (i) Define $\ol{\la} := \sup_{n \ge 1}\la_n$ and $z'_k = \la_k\ol{\la}^{-1}z_k \sim \CE(\ol{\la})$. We have: 
$z_1 + \ldots + z_n \ge z'_1 + \ldots + z'_n$. By the Law of Large Numbers, 
$z'_1 + \ldots + z'_n = n\ol{\la}^{-1}(1 + o(1))$ as $n \to \infty$. Therefore, we can estimate the infinite series as
$$
\SL_{n=1}^{\infty}e^{-\al(z_1 + \ldots + z_n)^2} \le \SL_{n=1}^{\infty}e^{-\al(z'_1+\ldots + z'_n)^2} \le \SL_{n=1}^{\infty}e^{-\al(\ol{\la}^{-2}(1 + o(1))n^2} < \infty.
$$

%
%

(ii) Recall that $\Var z_n = \la_n^{-2}$. For $S_n := z_1 + \ldots + z_n,\ n \ge 1$, we have: $\ME S_n = \La_n$. By \cite[Theorem 1.4.1]{StroockBook}, we have: $S_n - \La_n$ is bounded. The rest is trivial.

\subsection{Proof of Theorem~\ref{limitpointthm}} (i) It suffices to show that for every $k = 1, 2, \ldots$, the family of real-valued random variables
$$
Z_k = (Z_k(t), t \ge 0)
$$
is tight in $\BR_+$. Find an $N_j > k$ such that $[R^{(N_j)}]^{-1}\mu^{(N_j)} < 0$. Consider a finite system of $N_j$ competing Brownian particles with parameters 
$$
(g_n)_{1 \le n \le N_j},\ (\si_n^2)_{1 \le n \le N_j},\ (q^{\pm}_n)_{1 \le n \le N_j}.
$$
Denote this system by $Y^{(N_j)}$, as in the proof of Theorem~\ref{thm2}. Let 
$$
Z^{(N_j)} = (Z^{(N_j)}_1, \ldots, Z^{(N_j)}_{N_j-1})'
$$
be the corresponding gap process. By Proposition~\ref{basic}, the family of $\BR^{N_j-1}_+$-valued random variables $Z^{(N_j)}(t), t \ge 0$, is tight in $\BR^{N_j-1}_+$. By \cite[Corollary 3.9, Remark 9]{MyOwn2},
$$
Z^{(N_j)}_k(t) \ge Z_k(t) \ge 0,\ \ k = 1, \ldots, N_j-1.
$$
Since the collection of real-valued random variables $Z^{(N_j)}_k(t)$, $t \ge 0$, is tight, then the collection $Z_k(t), t \ge 0$, is also tight. 

\medskip

(ii) Fix $M \ge 2$. It suffices to show that 
$[\nu]_{M} \preceq [\pi]_{M}$. Since 
$[\pi^{(N_j)}]_{M}\ \Ra\ [\pi]_{M}$, as $j \to \infty$, 
it suffices to show that for $N_j > M$, we have: $[\nu]_{M} \preceq [\pi^{(N_j)}]_{M}$. Consider the system 
$$
Y^{(N_j)} = \left(Y^{(N_j)}_1, \ldots, Y^{(N_j)}_{N_j}\right)',
$$
which is defined in Definition~\ref{defapprox}. Let $Z^{(N_j)}$ be the corresponding gap process. Then 
$$
Z^{(N_j)}(t)\ \Ra\ \pi^{(N_j)},\ \ t \to \infty.
$$
But by \cite[Corollary 3.9, Remark 9]{MyOwn2}, 
$Z^{(N_j)}_k(t) \ge Z_k(t)$, $k = 1, \ldots, N_j - 1$. 
Therefore, $[Z^{(N_j)}(t)]_{M} \ge [Z(t)]_{M}$, for $t \ge 0$. 
And $[Z(t_j)]_{M}\ \Ra\ [\nu]_{M}$, as $j \to \infty$. Thus, $[\pi^{(N_j)}]_M \succeq [\nu]_M$. 

\medskip

(iii) Follows directly from (i). 

\subsection{Proof of Theorem~\ref{thmasymminf}} The proof resembles that of Lemma~\ref{notieinf} and uses Lemma~\ref{aux}. 

\medskip

(i) Define the following events:
$$
D = \{\exists t > 0:\ \exists k < l:\ Y_k(t) = Y_{k+1}(t),\ Y_l(t) = Y_{l+1}(t)\};
$$
$$
D_{k, l} = \{\exists t > 0:\ Y_k(t) = Y_{k+1}(t),\ Y_l(t) = Y_{l+1}(t)\}\ \ \mbox{for}\ \ k < l.
$$
Then it is easy to see that 
$$
D = \bigcup\limits_{k < l}D_{k, l}.
$$
Suppose $\omega \in D_{k, l}$, and take the $t  = t(\omega) > 0$ such that $Y_k(t) = Y_{k+1}(t)$, and $Y_l(t) = Y_{l+1}(t)$. 
There exists an $m > l$ such that $Y_l(t) = Y_{l+1}(t) = \ldots = Y_m(t) < Y_{m+1}(t)$, because otherwise the system $Y$ is not locally finite. Then there exist rational $q_-, q_+$ such that 
$$
t \in [q_-, q_+],\ \ \mbox{and}\ \ Y_m(s) < Y_{m+1}(s)\ \ \mbox{for}\ \ 
s \in [q_-, q_+].
$$
Therefore, $L_{(m, m+1)}(t) = \const$ on $[q_-, q_+]$, and, as in Lemma~\ref{notieinf},
$$
\left(\left(Y_1(s + q_-), \ldots, Y_m(s + q_-)\right)', 0 \le s \le q_+ - q_-\right)
$$
is a ranked system of $m$ competing Brownian particles with drift coefficients $(g_k)_{1 \le k \le m}$, diffusion coefficients $(\si_k^2)_{1 \le k \le m}$, and parameters of collision $(q^{\pm}_k)_{1 \le k \le m}$. This system experiences a simultaneous collision at time $s = t - q_- \in (0, q_+ - q_-)$. By \cite[Theorem 1.1]{MyOwn3}, this event has probability zero. Let us write this formally. Let 
\begin{align*}
D_{k, l, q_-, q_+, m} & = \{\exists t \in (q_-, q_+):\ Y_k(t) = Y_{k+1}(t),\ Y_l(t) = \ldots = Y_m(t) < Y_{m+1}(t), \\ & \mbox{and}\ \ Y_m(s) < Y_{m+1}(s)\ \ \mbox{for}\ \ s \in (q_-, q_+)\}.
\end{align*}
Then
$$
D = \bigcup\limits_{k < l}D_{k, l} \subseteq \bigcup D_{k, l, q_-, q_+, m},
$$
where the latter union is taken over all positive integers $k < l < m$ and positive rational numbers $q_- < q_+$. This union is countable, and by \cite[Theorem 1.2]{MyOwn3}, 
$\MP(D_{k, l, q_-, q_+, m}) = 0$, for each choice of $k, l, m, q_-, q_+$. Therefore, $\MP(D) = 0$, which completes the proof of (i).

\medskip

(ii) Let $B_1, B_2, \ldots$ be the driving Brownian motions of the system $Y$. Consider the ranked system of three competing Brownian particles:
$$
\ol{Y} = \left(\ol{Y}_{k-1}, \ol{Y}_k, \ol{Y}_{k+1}\right)',
$$
with drift coefficients $g_{k-1}, g_k, g_{k+1}$, diffusion coefficients $\si_{k-1}^2, \si_k^2, \si_{k+1}^2$ and parameters of collision $q^{\pm}_{k-1}, q^{\pm}_k, q^{\pm}_{k+1}$, with driving Brownian motions $B_{k-1}, B_k, B_{k+1}$, starting from 
$$
(Y_{k-1}(0), Y_k(0), Y_{k+1}(0))'.
$$
Let $(\ol{Z}_{k-1}, \ol{Z}_k)'$ be the corresponding gap process. Then by  \cite[Corollary 3.10, Remark 9]{MyOwn2}, we get:
$$
Z_{k-1}(t) \le \ol{Z}_{k-1}(t),\ \ Z_k(t) \le \ol{Z}_k(t),\ \ t \ge 0.
$$
But by \cite[Theorem 2]{MyOwn3}, with positive probability there exists $t > 0$ such that $\ol{Y}_{k-1}(t) = \ol{Y}_k(t) = \ol{Y}_{k+1}(t)$. So $\ol{Z}_{k-1}(t) = \ol{Z}_k(t) = 0$. Therefore, with positive probability there exists $t > 0$ such that $Z_{k-1}(t) = Z_k(t) = 0$, or, in other words, $Y_{k-1}(t) = Y_k(t) = Y_{k+1}(t)$.

\section{Appendix: Technical Lemmata}

\begin{lemma}
Assume that $(y_n)_{n \ge 1}$ is a sequence of real numbers such that 
$$
y_n \to \infty\ \ \mbox{and}\ \ \SL_{n=1}^{\infty}e^{-\al y_n^2} < \infty\ \ \mbox{for}\ \ \al > 0.
$$
Then for every $v \in \BR$ and $\be > 0$ we have:
$$
\SL_{n=1}^{\infty}\Psi\left(\frac{y_n + v}{\be}\right) < \infty.
$$
\label{lemma:special-series}
\end{lemma}

\begin{proof}
By \cite[Chapter 7, Lemma 2]{FellerBook}, we have for $v \ge 1$:
$$
\Psi(v) \le \frac1{\sqrt{2\pi}v}e^{-v^2/2} \le \frac1{\sqrt{2\pi}}e^{-v^2/2}.
$$
But $y_n \to \infty$ as $n \to \infty$, and there exists $n_0$ such that for $n \ge n_0$ we have: $(y_n+v)/\be \ge 1$. Therefore, for $n \ge n_0$, we have:
$$
\Psi\left(\frac{y_n+v}{\be}\right) \le \frac1{\sqrt{2\pi}}\exp\left(-\frac1{2\be^2}(y_n+v)^2\right).
$$
Using an elementary inequality $(c+d)^2 \ge c^2/2 - d^2$ for all $c, d \in \BR$, we get:
$$
\frac1{2\be^2}(y_n+v)^2 \ge \frac1{4\be^2}y_n^2 - \frac1{2\be^2}v^2.
$$
Thus, 
\begin{align*}
\SL_{n>n_0}&\Psi\left(\frac{y_n+v}{\be}\right) \le \frac1{\sqrt{2\pi}}\SL_{n>n_0}\exp\left(-\frac{y_n^2}{4\be^2} + \frac{v^2}{2\be^2}\right) \\ & = 
\frac1{\sqrt{2\pi}}\exp\left(\frac{v^2}{2\be^2}\right)\SL_{n>n_0}\exp\left(-\frac{y_n^2}{4\be^2}\right) < \infty. 
\end{align*}
\end{proof}

\begin{lemma} 
\label{lemma:estimate-special}
Take an It\^o process
$$
V(t) = v_0 + \int_0^t\be(s)\md s + \int_0^t\rho(s)\md W(s),\ \ t \ge 0,
$$
where $v_0 \in \BR$, $W = (W(t), t \ge 0)$, is a standard Brownian motion, $\be = (\be(t), t \ge 0)$ and $\rho = (\rho(t), t \ge 0)$, are adapted processes such that a.s. for all $t \ge 0$ we have the following estimates: $\be(t) \ge \ol{g}$, $|\rho(t)| \le \ol{\si}$. 
If $x \le v_0 + \ol{g}T$, then we have the following estimate:
$$
\MP\left(\min\limits_{0 \le t \le T}V(t) \le x\right) \le 2\Psi\left(\frac{v_0 - x - (\ol{g}T)_-}{\ol{\si}\sqrt{T}}\right).
$$
\end{lemma}

\begin{proof} Let $M(t) = \int_0^t\rho(s)\md W(s),\ t \ge 0$. Then $M = (M(t), t \ge 0)$ is a continuous square-integrable martingale with $\langle M\rangle_t = \int_0^t\rho^2(s)\md s$. There exists a standard Brownian motion $B = (B(t), t \ge 0)$ so that we can make a time-change: $M(t) \equiv B(\langle M\rangle_t)$. Then 
\begin{align*}
\left\{\min\limits_{0 \le t \le T}V(t) \le x\right\} &\subseteq \left\{\min\limits_{0 \le t \le T}M(t) - (\ol{g}T)_- + v_0 \le x\right\} \\ & \subseteq \left\{\min\limits_{0 \le t \le T}B(\langle M\rangle_t) \le x - v_0 + (\ol{g}T)_-\right\}.
\end{align*}
Because $\langle M\rangle_t \le \ol{\si}^2T$ for $t \in [0, T]$, we have:
$$
\left\{\min\limits_{0 \le t \le T}B(\langle M\rangle_t) \le x - v_0 + (\ol{g}T)_-\right\} \subseteq \left\{\min\limits_{0 \le t \le \ol{\si}^2T}B(t) \le x - v_0 + (\ol{g}T)_-\right\}.
$$
Finally,
\begin{align*}
\MP\left(\min\limits_{0 \le t \le \ol{\si}^2T}B(t) \le x - v_0 + (\ol{g}T)_-\right) & =
2\MP\left(B(\ol{\si}^2T) \le x - v_0 + (\ol{g}T)_-\right) \\ & = 2\Psi\left(\frac{v_0 - x - (\ol{g}T)_-}{\ol{\si}\sqrt{T}}\right).
\end{align*}
\end{proof}

\begin{lemma}
\label{lemma:estimate-enhanced}
Assume that in the setting of Lemma~\ref{lemma:estimate-special}, we have $
|\be(t)| \le \ol{g}$ and $|\rho(t)| \le \ol{\si}$ for $t \ge 0$ a.s. If $x \ge |v_0| + \ol{g}T$, then 
$$
\MP\left(\max\limits_{0 \le t \le T}|V(t)| \le x\right) \le 4\Psi\left(\frac{v_0 - x - \ol{g}T}{\ol{\si}\sqrt{T}}\right).
$$
\end{lemma}

\begin{proof} This follows from applying Lemma~\ref{lemma:estimate-special} twice: once for the minimum and once for the maximum of the process $V$. (We can adjust Lemma~\ref{lemma:estimate-special} to work for maximum of $V$ in an obvious way.) 
\end{proof}

\begin{lemma}
\label{lemma:tight-martingale}
Take a sequence $(M_n)_{n \ge 1}$ of continuous local martingales on $[0, T]$, such that $M_n(0) = 0$, and $\langle M_n\rangle_t$ is differentiable for all $n$, and
$$
\sup\limits_{n \ge 1}\sup\limits_{t \in [0, T]}\frac{\md \langle M_n\rangle_t}{\md t} = C < \infty.
$$
Then the sequence $(M_n)_{n \ge 1}$ is tight in $C[0, T]$. 
\end{lemma}

\begin{proof} Use \cite[Chapter 2, Problem 4.11]{KSBook} (with obvious adjustments, because the statement in this problem is for $\BR_+$ instead of $[0, T]$). We need only to show that 
\begin{equation}
\label{1000}
\sup\limits_{n \ge 1}\ME(M_n(t) - M_n(s))^4 \le C_0(t - s)^2
\end{equation}
for all $0 \le s \le t \le T$ and for some constant $C_0$, depending only on $C$ and $T$. By the Burkholder-Davis-Gundy inequality, see \cite[Chapter 3, Theorem 3.28]{KSBook}, for some absolute constant $C_4 > 0$ we have:
\begin{equation}
\ME(M_n(t) - M_n(s))^4 \le C_4\ME\left(\langle M_n\rangle_t - \langle M_n\rangle_s\right)^2 \le C_4(C^2(t - s))^2 = C_4C^4(t - s)^2.
\label{0001}
\end{equation}
\end{proof}

\section*{Acknoweldgements}

I would like to thank \textsc{Ioannis Karatzas}, \textsc{Soumik Pal}, \textsc{Xinwei Feng}, \textsc{Amir Dembo}, and \textsc{Vladas Sidoravicius} for help and useful discussion. I am also thankful to anonymous referees for meticulously reviewing this article, which helped to significantly improve it. This research was partially supported by  NSF grants DMS 1007563, DMS 1308340, DMS 1405210, and DMS 1409434.

%

\end{document}